\documentclass[10pt]{amsart}
\usepackage{amssymb,amsmath}
\usepackage{hyperref}

\usepackage{bm}
\usepackage[]{graphicx}
\usepackage{color}
\usepackage{amssymb, epsfig}
\numberwithin{equation}{section}

\numberwithin{figure}{section}
\newtheorem{thm}{Theorem}[section]
\newtheorem{lem}[thm]{Lemma}

\newtheorem{defin}[thm]{Definition}
\newtheorem{andefin}[thm]{Ansatz}

\newtheorem{rem}[thm]{Remark}

\newcommand\R{{\mathbb R}}

\newcommand\1{{1\kern-.25em\hbox{\rm I}}}
\newcommand\eu{{1\kern-.25em\hbox{\sm I}}}

\newcommand\FF{{\mathcal F}}

\newcommand\LL{{\mathcal L}}

\newcommand\NN{{\mathcal N}}
\newcommand\MM{{\mathcal M}}

 \newcommand\e{\varepsilon}
\newcommand\eps{\varepsilon}

\newcommand\s{\sigma}

\def\dse{{\rm d}S_{\eta}}
\def\dsx{{\rm d}S_{\xi}}
\def\NExt{{\mathcal E}_{\G,N}}

\def\NExt{{\mathcal E}_{\G,N}}

\def \sG { \lower 8pt \hbox {$\scriptstyle \G$}}
\def\d{{\delta}}
\def\t{{\tau}}

\newcommand\G{\Gamma}
\newcommand\Om{\Omega}

  \newcommand{\eightpoint}{\scriptstyle }
\newcommand{\nada}[1]{}

\begin{document}
%
%
%
%
%
\title[Sharp interface limit for the generalized Cahn-Hilliard]
{A Hilbert expansions method for the rigorous sharp interface
limit of the generalized Cahn-Hilliard Equation}
\author[Antonopoulou, Karali, Orlandi]{D.C.~Antonopoulou$^{\dag\ddag}$,
G.D.~Karali$^{\dag\ddag}$, E. Orlandi$^*$}
%
%
%
%
%
%
\thanks
{$^*$ Dipartimento di Matematica e Fisica, Universit\'a   di Roma Tre, L.go S. Murialdo 1, 00146 Roma, Italy.   orlandi@mat.uniroma3.it}
\thanks
{$^{\dag}$ Department of Applied Mathematics, University of Crete, 
GR--714 09 Heraklion, Greece. gkarali@tem.uoc.gr, danton@tem.uoc.gr}
\thanks
{$^{\ddag}$ Institute of Applied and Computational Mathematics,
FO.R.T.H., GR--711 10 Heraklion, Greece. } 
%
%
%

%
%
\begin{abstract}
We consider   Cahn-Hilliard equations with  external forcing terms.  Energy decreasing and mass conservation  might not hold. We show that level surfaces of  the solutions of such  generalized   Cahn-Hilliard equations   tend to the solutions of a moving boundary problem under the assumption that classical solutions of the latter exist.  Our strategy is to     construct  approximate solutions of the generalized Cahn-Hilliard equation
  by the  Hilbert expansion  method used in kinetic
theory and  proposed for the standard Cahn-Hilliard equation, by
Carlen, Carvalho and Orlandi, \cite {CCO}. The constructed  approximate solutions allow to  derive rigorously the  sharp interface limit of the generalized   Cahn-Hilliard equations. 
We then estimate the difference between the true solutions and the approximate solutions by spectral analysis, as in \cite {A-B-C}.

\end{abstract}
\maketitle \textbf{Keywords:} {\small{Cahn-Hilliard equation,
forcing, sharp interface limit, Hilbert expansion.}}
%
%
\pagestyle{myheadings}
\thispagestyle{plain}
%
%
%
\section{Introduction}\label{sec1}
    In this paper, we apply an alternative method to matched asymptotic expansions, developed  by  Carlen,
Carvalho and Orlandi, in  \cite {CCO}, which allows  the study of
the sharp interface limit for the generalized  Cahn-Hilliard
equation, and derive higher order corrections to this limit. The
method is based on the Hilbert expansion used in kinetic theory;
we refer to \cite {CCO} where the analogy is explained.   We start
by recalling some back ground regarding the  Cahn-Hilliard equation
and  the results obtained in   \cite {CCO}.
 
\subsection{The  Cahn-Hilliard Equation}\label{subsec1}  Let  $\Omega$  be  a
bounded domain in $\mathbb{R}^2$. The restriction of the analysis
to two dimensions is made only for simplicity.
 Let $m=m(x,t)$ be an integrable function on $\Om$ which
represents the value of a {\it conserved ``order parameter''} at
$x$ in $\Om$ at time $t$. The order parameter is conserved in the
sense that $\int_{\Om} m(x,t){\rm d}x$ is independent of $t$.
Therefore, the evolution equation for $m$ can be written in the
form
\[
  \partial_t m(x,t) = \nabla \cdot \vec J(x,t),
\]
where the {\it current} $\vec J$ is orthogonal to the outer normal
of the boundary of $\Om$. We take
\[
\vec J(x,t) = \sigma(m(x,t))\nabla \mu(x,t),
\]
where  $\sigma(m)$ is the {\it mobility} and $\mu(x,t)$ is the
{\it chemical potential} of $x$ at time $t$. The mobility is
positive and the chemical potential is defined as the $L^2(\Om)$
Frechet derivative of a {\it free energy functional}
 $\FF$:

$$
\mu(x) = \frac {\d \FF} {\d m}(x).
$$
 The simplest and most familiar example is the so called Cahn--Hilliard
equation. It results by setting $\sigma(m) := 1$, i.e. constant
mobility, and
\[\FF (m) :=\frac 12 \int_\Om |\nabla m(x)|^2{\rm d}x +
\frac 14 \int_\Om (m^2(x) -1)^2{\rm d}x.
\]
This leads to the evolution equation
\[
 \partial_t m(x,t) = \Delta\left(-\Delta m(x,t) +
f(m(x,t))\right),
\]
where
\begin{equation}
\label{L.1} f(m) = m^3 - m.
\end{equation}
Different choices of $f$ can be made, provided they are
derivatives of a double well smooth enough potential with equal
absolute minima. If $m(x,t)$ is a solution of this equation, then
\[
{{\rm d} \over {\rm d} t} \FF (m(\cdot, t)) = -\int_\Om  |\vec J(x,t)|^2{\rm d}x,
\]
and thus, evolution  decreases the free energy. The minimizers of
the free energy are the constant functions $m = \pm 1$. These
minimizers represent the ``pure phases'' of the system. However,
unless the initial condition $m_0$ happens to satisfy $\int_\Om
m_0(x){\rm d}x = \pm |\Om|$, these ``pure phases'' cannot be
reached due to the mass conservation law. Instead, what will
eventually be produced is a region in which $m \approx +1$ while
$m \approx -1$ in its complement, with smooth transition across
its boundary. This phenomenon is referred to as {\it phase
segregation}, where the aforementioned boundary consists the {\it
interface} between the two phases. If we ``stand far enough
back'' from $\Om$, all we can observe is the interface's shape
since the structure across the interface is placed on an invisibly
small scale.

The evolution of $m$ in time under the Cahn Hilliard equation, or
another equation of this type,  drives a very slow evolution of
the interface. More specifically, let $\e$ be a small parameter,
and introduce the new variables $\tau$ and $\xi$ by
\[
\tau := \e^3t \qquad{\rm and}\qquad \xi := \e x .
\]
Then of course it follows
\[
 \partial_t = \e^3   \partial_{\tau}
\qquad{\rm and}\qquad \partial_x = \e
 \partial_\xi\ .
\]
Hence, if $m(x,t)$ is a solution of the Cahn--Hilliard equation
and we define $m^\e(\xi,\tau) := m(x(\xi),t(\tau))$ then we obtain
\begin{equation}
\label{ch0}
 \partial_\tau  m^{\e} (\xi,\tau)    =
\Delta_\xi\left(-\e\Delta_\xi
m^\e(\xi,\t)
   +\frac 1 \e  f(m^{\e}  (\xi,\t) ) \right) .
\end{equation}
 If we think of $\e$ as representing the inverse of a large length
scale,  the variable $\xi$ will be dimensionless.  The
dimensionless variables are ``slow'' and the original variables
``fast'' for small $\e$. In what follows, we keep the notation
$\xi$ for the slow spatial variables, but we drop the use of
$\tau$ and replace it by $t$ for convenience. One should just
bear in mind that now we are looking at the evolution over a {\it
very} long time scale when $\e$ is small. For the reasons
indicated above, it is custom to  consider initial data $m_0(\xi)$
that is $-1$ in the region bounded by a smooth closed curve
$\G_0$ in $\Om$, and $+1$ outside this region.   At  later times
$t$ there will still be
 a fairly sharp interface between a region where $m(\xi,t) \approx +1$
and  a region where $m(\xi,t) \approx -1$, centered on a smooth
curve $\G_t$. One might hope that for small values of $\e$, {\it
all information about the evolution of $m^\e(\xi,t)$ is contained
in the evolution of the interface $\G_t$}. This is indeed the
case as shown in \cite {CCO}. To explain the method used in \cite {CCO}, let ${\mathcal M}$
denote the set of all smooth simple closed curves in $\Om$. As we
will explain in Section 2,   ${\mathcal M}$ can be viewed as a
differentiable manifold.  A
 vector field $V$ on ${\mathcal M}$ is a functional associating to each
$\G$ in ${\mathcal M}$ a function in $C^\infty(\G)$. This function gives
the normal velocity of a point on $\G$, and thus describes a
``flow'' on ${\mathcal M}$. We may formally write
\begin{equation}
\label{flow1a} {{\rm d}\over {\rm d}t}\G_t = V(\G_t) .
\end{equation}
Now, given a flow on ${\mathcal M}$, we can produce from it an
evolution in $C^\infty(\Om)$ through the following device: Let
$m$ be any function from ${\mathcal M}$ to  $C^\infty(\Om)$. We
write $m(\xi,\G)$ to denote $m(\G)$ evaluated at $\xi\in \Om$. We
can then define a time dependent function $m(\xi,t)$ on $\Om$ by
\begin{equation}
\label{flow2a} m(\xi,t) := m(\xi,\G_t) .
\end{equation}
Notice that time dependence in $m(\xi,t)$ enters {\it only} through
the evolution of $\G_t$.
 Now if, for small $\e$ and sharp interface initial data, all of the
information about the evolution of solutions of the Cahn--Hilliard
equation were contained in the motion of the interface, then one
might hope to find a vector field $V$ on ${\mathcal M}$ governing the
evolution of the interface, and a function $m$  from ${\mathcal M}$ to
$C^\infty(\Om)$ so that \eqref {flow2a} defines the corresponding solution
of the Cahn--Hilliard equation.

In  \cite {CCO}, a result of this type has been proved.  More
specifically, a sequence of vector fields $V_0,V_1,V_2,\cdots$
defined on ${\mathcal M}$ was constructed such that the interface
for the solution of \eqref {ch0} satisfies \eqref {flow1a} for $V
:= \sum_{j=0}^\infty \e^jV_j$.  It turned out that the leading
term $V_0$  is the vector field generating the {\it Mullins
Sekerka} flow, as one could expect from the pioneering work of
Pego \cite{Pego}  made rigorous by Alikakos, Bates and Chen   \cite {A-B-C}.  
In these papers 
the approximate solutions were constructed by using matched
asymptotic expansions which give no information on the higher
order corrections to the flow. The approach introduced in \cite
{CCO} unable   to determine at any given order the velocity of
the flow.

Let us fix a number $S>0$ that will later be interpreted as a
``surface tension', denote by $K(\xi)\equiv K(\xi,\G)$  the
curvature at $\xi \in \G$ and by $\nu$ the unit outward normal either to $ \partial \Om$ or to $ \G$.
Further, for each $\G$ in ${\mathcal
M}$, let $\mu $ be the solution of
\begin{equation}
\label{ms1a} \Delta\mu (\xi) = 0 \qquad{\rm for}\qquad \xi \in
\Om\setminus \G,
\end{equation}
subject to the boundary conditions
\begin{equation}
\label{ms2a} \mu (\xi) = S\left(K(\xi)-\frac { 2\pi}
{|\G|}\right)\;\;\;\;{\rm on}\;\;\;\;\G,\;\;\;\;
  \partial_\nu \mu  = 0\;\;\;\;{\rm on}\;\;\;\;
\partial\Om,
\end{equation}
where $|\G|$ denotes the arc length of $\G$  and $\partial_\nu$ is the outward normal
derivative to $\partial \Om$.
 Now define $V_0(\G)$ as the real valued function on $\G$ given by
\begin{equation}
\label{ms3a} V_0(\xi,\G) :=  \frac 12  \left [  \partial_\nu  \mu \right]_{\G }(\xi) \qquad \xi \in \G,
\end{equation}
where the brackets on the right-hand side denote the jump of the
normal derivative across $\G $.  In this way one defines a vector
field on ${\mathcal M}$ which generates a flow known as the
Mullins--Sekerka flow.
For the local existence of a unique smooth solution
of the free boundary problem \eqref {ms1a}, \eqref {ms2a} and
\eqref {ms3a} see  Chen,\cite{C2}   and    Escher and Simonett, \cite{ES}. 
As it is well known, the Mullins--Sekerka flow conserves the area
enclosed by $\G_t$ and decreases the arc length of $\G_t$.

The higher order terms in  $\sum_{j=0}\e^{j}V_j$ are more
complicated. In  \cite {CCO}, $V_1$ which is the next correction
to $V_0$ was computed and a general technic of calculating all the
higher order terms has been presented. The description of $V_1$,
like that of $V_0$, is in the context of potential theory.

\subsection{The $\varepsilon$-dependent generalized Cahn-Hilliard
Equation}\label{subsec2}
We consider the generalized  Cahn-Hilliard equation of the
following type
\begin{equation}\label{1a}
\begin{split}
& \partial_t  m^{\eps}(\xi,t)=\Delta\Big{(}-\eps
\Delta
m^\eps(\xi,t)+\frac{f(m^{\eps}(\xi,t))}{\eps}-G_2(\xi;\eps)\Big{)}
+G_1(\xi;\eps),\;\;\;\xi\;\;\mbox{in}\;\;\;\Omega,\;\;\;t> 0,
\end{split}
\end{equation}
where $\Delta$ is the  Neumann Laplacian  operator on $\Omega$.
The terms $G_1$ and $G_2$ may depend on time also. In the present
analysis, we shall consider the case where $G_1$ and $G_2$ depend
only on $\xi$ since our aim is to explain the main strategy in
the simplest interesting setting. As it will be clear in the
sequence, the proposed method is suitable for the time dependent
case as well.

The term $G_2$ in \eqref{1a} models general external fields,  see
\cite{Hohen,Gurtin}. In  \cite{Kitahara}  the authors apply the
Kawasaki  exchange dynamics to derive a modified Cahn-Hilliard
equation where $G_2$ describes the external gravity field. The
free energy-independent term $G_1$ may describe an external mass
supply, cf. \cite{Gurtin}, or \cite{Antal} where $G_1$ was
defined as a deterministic Gaussian function.  Such a model is
described for example in \cite{Antal}, in order to model spinodal
decomposition in the presence of a moving particle source, as a
mechanism for the formation of Liesengang bands.  
In addition, $G_1$ was introduced as a conservative white noise
of thermal fluctuations cf. \cite{Hohen} or \cite{cook}
(Cahn-Hilliard-Cook model).  
Existence
and uniqueness of solution for the stochastic problem was
established in \cite{Debu2,Weber,AK} while dynamics and
stochastic stability were analyzed for the one-dimensional case
in \cite{ABK}. Furthermore, the interface stochastic motion and
singular perturbation has been studied for many related models
like Allen-Cahn or Ginzburg Landau and phase-field models, cf.
for example \cite{Fun1,abak}.

Integrating  \eqref{1a}  over $\Omega$ we
get
 \begin{equation}\label{genlowmass2}
\partial_{t}\left ( \int_\Om m^\e (\xi,t ) d \xi\right)=\int_{\partial\Omega}\partial_{\nu} \Big{(}-\eps \Delta m^\eps +
\frac{f(m^\eps)}{\eps}-G_2(s;\eps)\Big{)}ds+\int_{\Omega}G_1(\xi;\eps)d\xi.
\end{equation}
Therefore,   $\int_\Om m^\e (\xi,t ) d \xi $ is not conserved unless
the  second member of \eqref {genlowmass2} is  null.
 Generally, due to the
presence of the external force field $G_2$ and the external mass
supply $G_1$, a free energy decreasing is not expected.   For a
mathematical analysis of the problem when $G_2=0$ and $G_1$ is in
$L^2(\Omega)$ cf. \cite{C-H-forcing,El-Z}.

An equivalent system formulation of \eqref{1a} is the following
\begin {equation}\label {peg1}  \partial_t   m^\e (\xi,t)    =
   \Delta
\mu^{\e} (\xi, t ) +   G_1(\xi;\e),  \end {equation}
\begin {equation} \label{peg2}
 \mu^{\e}(\xi, t )  = -\e \Delta
m^\e(\xi,t)
   +\frac 1 \e  f(m^\e  (\xi,t) ) -   G_2(\xi;\e),      \end {equation}
where $\Delta$ is the  Neumann Laplacian  operator on $\Omega$.
This representation will be used in our analysis. For the
purposes of this paper we consider the $\eps$-dependent
generalized Cahn-Hilliard equation \eqref{1a} (and equivalently
the system \eqref {peg1}, \eqref {peg2}) supplemented with an
initial condition
\begin{equation}\label{initial1}
m^\eps(\xi,0)=m^\eps_0(\xi) \simeq  \left \{\begin {split} & -1  \quad \hbox {on} \quad \Om^-_0\cr &
+1 \quad \hbox {on} \quad \Om^+_0 , \end {split} \right.
\end{equation}
where    $\Om^-_0$ is the region of $\Omega$ enclosed by a smooth closed curve $\G_0$
 and $ \Om^+_0 = \Om \setminus \left ( \Om^-_0 \cup \G_0 \right)$.
 Thus, we are assuming  that the interface is already initial formed. 
 Further we take the following
Neumann boundary conditions
\begin{equation}\label{bcgen}
\partial_\nu  m^\eps=\partial_\nu  \Delta m^\e= 0\;\;\;\mbox{on}\;\;\;\partial\Omega.
\end{equation}
 We assume that the forcing  terms $G_1$ and $G_2$ are sufficiently smooth,
and that   $ \partial_ \nu G_2 =0$ on $ \partial \Om$ so that  \eqref {bcgen} becomes
\begin{equation}\label{bcgen1}
\partial_\nu  m^\eps=\partial_\nu  \mu^\e=\partial_\nu  \Delta m^\e= 0\;\;\;\mbox{on}\;\;\;\partial\Omega.
\end{equation}
We dot not  require
\begin{equation}\label{masscons}
 \int_{\Omega}G_1(\xi;\eps)d\xi=0.
\end{equation}
Hence, mass conservation  might not hold. The precise assumptions
for  the forcing terms $G_1$ and $G_2$ will be given in Section
2. For sufficiently smooth initial conditions and forcing terms
$G_1$, $G_2$, there exists a unique  classical  solution of the
generalized Cahn-Hilliard equation. The proof is analogous to
that of the homogeneous case presented in \cite{El-Z}.
 
  Notice that   if we write \eqref {1a} in the original  not scaled  variables  $(x,t)$
 the terms $G_1$ and $G_2$ are  
 small perturbations of the standard Cahn-Hilliard
equation. The term $G_1$ in the original variables $(x,t)$ is
multiplied by a factor $\e^3$ and the term $G_2$ by a factor $\e$.
The problem that we pose is the following. Take  as in the
homogeneous Cahn-Hilliard equation, initial data $m_0(\xi)$ like in \eqref {initial1}.
Due to the presence of $G_1$ and $G_2$ the constant functions
$m^\e = \pm 1$ are not anymore stationary solutions of \eqref
{1a}. But we still expect that eventually at later times $t$ there
will appear a fairly sharp interface between the regions where
$m^\e(\xi,t) \approx +1$ and where $m^\e(\xi,t) \approx -1$,
centered on some smooth curve $\G_t$.
 We prove that this is
indeed the case.
We derive the motion of $\G_t$
determining the vector field. It turns out that the leading term
$V_0$ in the vector field $\sum_{j=0}^{N-1} \e^j V_j
\left(\G^{(N)}_t \right)$, governing the interfacial flow (see
\eqref {e0}), is not the vector field generating the Mullins
Sekerka  flow appearing in the sharp limit of the homogeneous
Cahn-Hilliard equation. In fact, we obtain
  \begin{equation}\label{nov2} V_0 (\cdot,\G_t)= V^{(0)}_0(\cdot,\G_t) + \langle V_0\rangle_{\eightpoint
\Gamma_t}, \end {equation} where
 $$  \int_{\G_t}V^{(0)}_0(\eta ,\G_t) \dse =0, $$
 and
\begin{equation}
\label{nov.4}  \langle V_0\rangle_{\eightpoint \Gamma_t}
 =   \frac 1 {2 |\G_t|} \int_\Om  G_{1,0} (\eta)  {\rm d}\eta\;\;\;\;  t \in [0,T]. \end {equation}
  Here, and in the following, we denote by $\dse$ the element
of the arc length along $\G$ or $\partial \Om $.   We will indeed  prove that, as $ \e \to
0$ the singular limit  of \eqref {peg1} and \eqref {peg2} leads
to the following  moving boundary problem: Given a closed curve
$\G^0$ in $\Omega$ that it is the boundary of an open set
$\Om^-_0 \subset \Om$ find a family $ \Big{\{} \G_t \in \MM: t \in
[0,T] \Big{\}}$  and functions $\mu(\xi,t)=\mu(\xi,\G_t) $ for $t
\in [0,T]$ and $\xi \in \Om$   so that
\begin{equation}
\label{nov.3}  \begin {split} & \Delta\mu (\xi,t) = -G_{1,0}
(\xi)   \qquad \xi \in \Om\setminus \G_t,\qquad  t \in  [0,T), \cr
& \mu (\xi,t) =  2S K(\xi,\G_t) -   G_{2,0} (\xi)  \quad \hbox
{\rm on}\quad \G_t,\;\;\;\;
  \partial_ \nu \mu (\cdot,t) = 0  \quad \hbox {\rm on}\quad
\partial \Om, \qquad  t \in  [0,T),  \cr &
 V_0(\cdot, \G_t) = \frac 12  \left [  \partial_
\nu   \mu \right]_{\G_t }(\xi) \qquad \xi \in \G_t, \qquad t \in
(0,T), \cr & \G_0 = \G^0,
\end {split}
\end{equation}
where $G_{1,0}(\xi):=\displaystyle{\lim_{\eps\rightarrow
0}}G_1(\xi;\eps)$ and
$G_{2,0}(\xi):=\displaystyle{\lim_{\eps\rightarrow
0}}G_2(\xi;\eps)$ and $S>0$ is the surface tension defined in \eqref {surf}.
In \cite{AKK} the authors
applied formal asymptotics to analyze the sharp interface motion
for  generalized Cahn-Hilliard equations of the form \eqref
{1a}. The limit problem, which was  formally
derived in \cite{AKK},
   agrees exactly to \eqref{nov.3} which is rigorously proven in
this paper.

We immediately obtain for any $t \in (0,T)$
  \begin{equation}
\label{ms4a}
2 \int_{\G_t} V_0(\eta,\G_t) \dse = \int_{\G_t}    \left [  \partial_ \nu   \mu
\right]_{\G_t}(\eta) \dse =    -  \int_{\Om\setminus \G_t}\Delta \mu
(\eta,t) {\rm d} \eta  =  \int_\Om  G_{1,0} (\eta)  {\rm d}\eta,
\end{equation}
i.e  \eqref {nov.4}. Recalling that   ${{\rm d}\over {\rm d}t}
|\Om_{\G_t}^-| =  \int_{\G_t} V_0(\eta,\G_t) \dse$, we obtain
that the area enclosed  by $\G_t$ is not conserved unless $
\int_\Om  G_{1,0} (\eta)  {\rm d}\eta=0$. Also, we have
\begin {equation} \label{nov.5a} \begin{split}
{ {\rm d}\over {\rm d}t}  |\G_t|& =    \int_{\G_t} K(\eta, \G_t)
V_0(\eta,\G_t)\dse= \frac 1 {2S}  \left ( \int_{\G_t} \mu
V_0(\eta,\G_t)    \dse  +  \int_{\G_t} V_0(\eta,\G_t) G_{2,0}
(\eta) \dse \right ).
 \end{split}
\end{equation}
Let us denote by $ \mu^\pm (\cdot, \G_t)$  the restriction of $
\mu (\cdot, \G_t)$ in $\Om^\pm_t$. It follows that
\begin {equation} \label{nov.6} \begin{split}
   2 \int_{\G_t}  \mu  V_0(\eta,\G_t)\dse & =  \int_{\G_t}   \mu^+ \partial_{\nu} \mu^+  \dse -
\int_{\G_t}   \mu^- \partial_{\nu} \mu^-  \dse  \cr &=
-  \int_{\Om^-_t} div ( \mu \nabla
\mu )   {\rm d} \xi -     \int_{\Om^+_t} div ( \mu \nabla
\mu )   {\rm d} \xi   \cr &=  -  \int_{\Om}   |\nabla
\mu|^2  {\rm d} \xi  - \int_{\Om \setminus \G_t} \mu \Delta \mu    {\rm d} \xi  =
 -  \int_{\Om}   |\nabla
\mu|^2  {\rm d} \xi  + \int_{\Om \setminus \G_t} \mu  G_{1,0} (\xi)  {\rm d}\xi.
\end{split}
\end{equation}
From these computations there is no reason to expect  that ${
{\rm d}\over {\rm d}t}  |\G_t|$  is not positive. So, even in the
case when the volume is conserved, i.e  when $ \int_\Om  G_{1,0}
(\eta)  {\rm d}\eta =0$ the length of the curve does not
decrease. The unknown  $ \Big{\{} \G_t \in \Om: t \in [0,T]
\Big{\}}$ and $\mu^\pm$  are coupled through the system \eqref
{nov.3}. However if the position and the regularity of the moving
boundary   $ \Big{\{} \G_t \in \Om: t \in [0,T] \Big{\}}$ is
known, the chemical potential $ \mu$  is  obtained by  solving at
each time  $ t \in [0,T)$ the elliptic boundary value problem
\begin {equation} \label{nov.3a}  \begin {split} & \Delta\mu (\xi,t) = -G_{1,0} (\xi)   \qquad \xi \in
\Om\setminus \G_t,\qquad  t \in  [0,T), \cr &  \mu (\xi,t) =  2S
K(\xi,\G_t) -   G_{2,0} (\xi)  \quad \hbox {\rm on}\quad
\G_t,\;\;\;\;
  \partial_\nu \mu (\cdot,t) = 0  \quad \hbox {\rm on}\quad
\partial \Om, \qquad  t \in  [0,T).
\end {split}
\end{equation}
In this sense we call a family  $\{\Gamma(t);t\in[0, T)\}$  of
surfaces a  solution of \eqref {nov.3}. To our  knowledge there
are no result regarding the existence and the uniqueness of
solution for the moving boundary problem of the type \eqref
{nov.3}. A modified Mullins Sekerka motion has been studied by
\cite {EN}, but it differs from \eqref {nov.3} either at the
presence of the term $-G_{1,0}$ which is replaced in \cite {EN}
by a specific function of  time  only, either at the presence of
$-G_{2,0}$ which does not appear in \cite {EN}. We think that a
method similar to the one used in \cite {EN} might be useful to
give   existence and uniqueness of  the classical solution of \eqref {nov.3}.
For the purposes of this paper, we assume that there exists a
unique classical solution of the free boundary problem \eqref
{nov.3}.

%
\section{Notations and Main results}\label{sec2}

\subsection{ Vector fields and flow on the   curve space }\label{subsec3}  Let $\MM $
denote the set of all smooth simple closed curves in $\Om \subset
\R^2$. To discuss motion in $\MM $ it is convenient to introduce
local coordinates in the neighborhood of any given $\G \in \MM $.
To this aim we define:
\medskip
  \begin {defin} \label {curv}
Let $K(\xi)=K(\xi,\Gamma)$ denote the curvature at a point
$\xi\in \Gamma$ for   $\Gamma\in\MM$.  We define
$$k(\Gamma):=\displaystyle{\max_{\xi\in\Gamma}}|K(\xi)|.$$
\end{defin}
We denote  by $d(\xi, \G)$ the signed distance of $\xi \in \Omega
$ from $\G$. We define $d<0$ when $\xi$ is inside $\G$  and $d>0$
when $\xi$ is outside $\G$. As long as  $d(\xi, \G) \le
\frac{1}{k(\Gamma)} $ there is a uniquely determined point  $\eta
\in \G$ such that $|\eta- \xi| = d(\xi, \G)$; this is the  point
in $\Gamma$  closest  to $\xi$. Therefore, for any $\eps_0$ such
that $0<\eps_0<\frac{1}{k(\Gamma)}$, let
$$\mathcal{N}(\eps_0)=\mathcal{N}(\eps_0,\Gamma):=\Big{\{}\xi\in
\Omega :|d(\xi,\Gamma)|\le \eps_0\Big{\}}.$$ There is a natural set
of coordinates in $\mathcal{N}(\eps_0)$.  Given $\xi \in
\mathcal{N}(\eps_0)$ we denote by $\rho$ the diffeomorphism
 $ \rho:  \NN(\e_0) \to  [-\e_0, \e_0] \times \Gamma$  defined by
  $ \rho(\xi)= ( d (\xi), s(\xi))$ (whenever this does not cause ambiguity we omit to write the explicit
  dependence of $\NN$ or $d$ on $\G$).
  We have that
  $$ \xi=   s(\xi)+  d\; \nu (s(\xi)),$$
  where $\nu(s(\xi))$ denotes the unit outward normal to $\G$ at $ s(\xi)$.
  For $d \in [-\e_0,\e_0]$ and $s \in \G$ let $  \alpha (d ,s) $  be the Jacobian of the local change of
variables $ \alpha(d,s)= \hbox {det} \frac {\partial
\rho^{-1}(d,s)} {\partial (d,s)}  $. A standard computation (cf.
\cite [appendix] {GT}) gives
  $  \alpha(d,s) = \prod_{i=1}^{n-1} \left (  1- d K_i(s) \right ), $
where  $K_i(s)$, $i=1,\dots,n$ are the principal  curvatures at $
s \in \G$,  in the direction $i$.  When $n=2$ we have
\begin{equation} \label {8.15a} \alpha(d,s) =   1- d K (s).  \end {equation}
In the sequel we identify functions of variable $\xi$ and
functions of variable $(d,s)$ in the domain $\NN(\e_0)$.
We denote by
$$z = \frac d \e $$ the stretched variable.

The introduced coordinates in  $ \NN(\e_0, \G)$ provide the means
to give $\MM $ the structure of a differentiable manifold and to
study motions in this manifold, see \cite [Section 2]{CCO}. A
vector field $V$ on $\MM $ is a functional associating to each
$\G$ in $\MM$ a function in $C^\infty(\G)$. This function defines
the normal velocity of a point on $\G$ and thus, describes a
``flow'' on $\MM $. More specifically, we may formally write
\begin{equation}
\label{flow1} {{\rm d}\over {\rm d}t}\G_t = V(\G_t).
\end{equation}
We denote the lifetime $T$ of the flow  \eqref{flow1}, starting
at $\G \in \MM$ as
\begin{equation}
\label{ll.2} T = \inf \Big{\{} t>0: k(\G_t) \le k_0\Big{\}},
\end{equation}
where $k_0$ is  any arbitrarily chosen positive number so that  $
k(\G) \le k_0<\infty$. If  $V(\cdot,\G) = K(\cdot,\G)$, the
curvature at $s\in \G$, one obtains the curve shortening {\it
flow by curvature}. When $V(\cdot,\G)$ is given by \eqref
{ms3a}   we  have the Mullins--Sekerka vector field, described in
the  introduction. When
  $V(\cdot,\G)$ is given by \eqref  {nov2} we have the flow characterizing the sharp
  interface motion studied in this paper.

A given flow on $\MM $ produces an evolution in $C^\infty(\Om)$
through the following device: Let $m$ be any function from
${\mathcal M}$ to  $C^\infty(\Om)$; we write $m(\xi,\G)$ to
denote $m(\G)$ evaluated at $\xi\in \Om$. Then a time dependent
function $m(\xi,t)$ may be defined on $\Om$ as follows:
\begin{equation}
\label{flow2} m(\xi,t): = m(\xi,\G_t).
\end{equation}

 There is an obvious but useful decomposition of vector fields on
$\MM$. Given a vector field $V$ on $\MM$ we may apply the
decomposition
\begin{equation}
\label{decp} V(\cdot,\G) = V^{(0)}(\cdot,\G) + \langle
V\rangle_{\eightpoint \Gamma},
\end{equation}
where
\begin{equation}
\label{avapart} \langle V\rangle_{\eightpoint \Gamma} := {1\over
|\G|}\int_\G V(\xi,\G){\rm d}S_\xi,
\end{equation}
 and
 \[
V^{(0)}(\cdot,\G) := V(\cdot,\G) - \langle V\rangle_{\eightpoint
\Gamma}.
\]
 Since $\langle V\rangle_{\eightpoint
\Gamma}$ is constant then  by its definition $V^{(0)}$ is
orthogonal to the constants in the $L^2(\Gamma)$ inner product
i.e satisfies
$$  \int_\G V^{(0)}(\eta,\G) \dse =0,$$
and therefore, it generates a volume preserving flow in the sense
that for any $t$ the area enclosed by $\Gamma=\Gamma_t$ is
constant.

Under the ansatz  given below, in this paper, we derive separate
equations for the components $V^{(0)}(\cdot,\G)$ and $\langle
V\rangle_{\eightpoint \Gamma}$ for each of the vector fields
$V_j$.
  \begin {andefin} \label {main}  Let
$V_0,V_1,V_2,\cdots$ be a sequence of vector fields on $\MM$ and
$m_0,m_1,m_2,\cdots$ functions from $\MM$ to $C^\infty(\Om)$.
 For  any  given initial interface $\G_0$ in $\MM$
 and all $N > 0$,  let  $\G^{(N)}_t$ be the solution of
 \begin{equation}\label{e0}
{{\rm d} \G^{(N)}_t \over {\rm d}t} = \left[\sum_{j=0}^{N-1} \e^j
V_j\right]\left(\G^{(N)}_t \right)\qquad {\rm with}\qquad
\G^{(N)}_0 = \G_0.
\end{equation}

 We define the function $m^{(N)}(\xi,t)$  by
 \begin{equation}\label{e1}
m^{(N)}(\xi,t) =  m_0  \Big{(}\frac {d(\xi,\G^{(N)}_t)} \e
\Big{)} + \sum_{j=1}^N \e^j m_j(\xi,\G^{(N)}_t),
\end{equation}
 and notice that
$m^{(N)}(\xi,t)$ depends on $t$ only through $\G^{(N)}_t$.

 We set
 \begin{equation}\label{e1a}
m_0(z):=r\Big{(}\frac{\eps}{\eps_0}z\Big{)}\bar
m(z)+\Big{(}1-r\Big{(}\frac{\eps}{\eps_0}z\Big{)}\Big{)}{\rm
sgn}(z), \end{equation} where $\bar m(z):={\rm tanh}(z/\sqrt{2})$
\footnote {The explicit form of  the solution is never used. We
will use only its qualitative properties.} defined for any
$z\in\mathbb{R}$ is the unique solution of the Euler-Lagrange
equation
$$ -m'' (z) +f(m (z)) =0, \qquad z \in \R, \quad  \lim_{z \to \pm \infty}  m(z)= \pm 1, $$
  and $r$ is a smooth even unimodal cut-off function, $r(u)=1$ for $|u|  < \frac 12$ and $r(u)=0$ for $u>1$.

In addition, let
\begin{equation}\label{e2}
m_j(\xi,\G^{(N)}_t)  :=   h_j \Big{(}{d(\xi,\G^{(N)}_t)\over \e},
s(\xi,\G^{(N)}_t)\Big{)}
     + \phi_j (\xi,\G^{(N)}_t), \quad \xi \in \Omega, \quad
     j=1,\cdots,N,
\end{equation}
  where $h_j$ are  $C^\infty (\Omega)$  functions  equal to $ 0 $ in $\Om \setminus \NN
  (\e_0)$ and when $d(\xi,\G^{(N)}_t)=0$. The functions $\phi_j$, $ j=1,\cdots,N $  are in $C^\infty (\Omega)$,
  satisfy the Neumann boundary conditions on $ \partial \Omega$ and admit a global Lipschitz bound,
  independent of $\e$,
  i.e.
  $$ \|\phi_j\|_{ Lip (\Om)} \le C \quad j=1,\cdots,N,$$
  where $C$ is a constant independent of $\e$.
\end {andefin}

{\bf  Notational convention} Below we denote by $m^{(N)}(\xi,t):= m^{(N)}(\xi,\G^{(N)}_t)$ and
$\mu^{(N-1)}(\xi,t):= \mu^{(N-1)}(\xi,\G^{(N)}_t)$.
If there is no ambiguity we write  $\G$ or $\G_t$  for $\G^{(N)}_t$.  Trough what follows, we write $C$ to designate a generic positive constant independent on $\e$. Its actual numerical value  may change from one occurrence to the next.

\begin {rem} The Ansatz \ref {main}  must be modified when $G_1$ and $G_2$ depend on time. 
 The   \eqref {e1}   should be replaced by
 $$     m^{(N)}(\xi,t) =  m_0  \Big{(}\frac {d(\xi,\G^{(N)}_t)} \e
\Big{)} + \sum_{j=1}^N \e^j m_j(\xi, t, \G^{(N)}_t). $$
Notice that
$m^{(N)}(\xi,t)$ depends now on $t$ not only   through $\G^{(N)}_t$.
One can verify that   the first order function $ m_0 $  keeps to depend  on $t$ only trough $\G^{(N)}_t$.
  \end {rem}

\subsection {Main results}\label{subsec4}
We start constructing a function $ m^{(N)} (\xi, \G_t)$, for $
\xi \in \Om$ and $ t \in [0,T]$ where $T$ is the lifetime of
\eqref  {e0} and show
that it is an approximate solution of \eqref {1a}. We make the
following assumptions on the forcing terms $G_1$ and $G_2$.

\medskip \noindent
{\bf  A1: Assumptions for $G_1$ and $G_2$.}
For any $N>1$ we require
\begin{equation}\label{GG1a} \begin {split}
&G_i(\xi;\e) = \sum_{j=0}^{N-1}  \e^j G_{i,j}(\xi) + \e^{N}
G_{i,N}(\xi,\e ), \qquad   |G_{i,N}(\xi,\e )| \le C,\cr &
 G_{i,j}  \in C^\infty (\Om)  \qquad j=1,
\cdots, N-1, \quad  i=1,2, \cr &
 \partial_{\nu} G_{2,j } =0 \quad \hbox {on} \quad \partial \Om,  \qquad j=1,
 \cdots,
 N.
\end {split} \end{equation}

\begin {rem}We require
$G_{i,j} \in C^\infty (\Om)$ for $j=1, \cdots, N-1$ and for
$i=1,2$ to  avoid regularity problems, but this assumption can be
relaxed.   \end {rem}

 \medskip
\noindent\begin{thm} \label{1.1} Let $N>1 $ and $G_1$ and $G_2$
be as in assumptions  A1.
    There exist vector fields $V_j$,
$j=0,\cdots,(N-1)$ and  functions $m_j$, $j=0,\cdots, N$ as
prescribed in the Ansatz \ref{main} having the following
properties: Let $T$ denote the lifetime of the solution of \eqref
{e0} in ${\mathcal M}$. Then there is a constant $C_N$ so that
for all $t < T$
\begin{equation}
\label{ch0N}
 \partial_t  m^{(N)} (\xi,t)    =
\Delta\left(-\e \Delta
m^{(N)} (\xi,t)
   +\frac 1 \e  f(m^{(N)}   (\xi,t) )
   - \sum_{j=0}^{N-1}  \e^j G_{2,j}(\xi)  \right) +   \sum_{j=0}^{N-1}  \e^j G_{1,j}(\xi)  + \Delta
   R^{(N)}
(\xi,t),
\end{equation}
where
\begin{equation}
\label{goodfit} \sup_{\xi\in \Om,t \in [0,T]}\left|R^{(N)}
(\xi,t)\right| \le  C_N\e^{N-1}.
\end{equation}
Finally, the sequences of vector fields and functions are
essentially uniquely determined since given $V_j$ for $j < k$
then $V_k$ is determined up to ${\mathcal O}(\e^{k+1})$, and
similarly, given $m_j$ for $j < k$ then $m_k$ is determined up to
${\mathcal O}(\e^{k+1})$.
\end{thm}

\begin{rem}\label{remth}
 In Theorem \ref {1.1} and in the following, the symbol ${\mathcal
O}(\e^{m})$ denotes terms which are of order $\e^m$ uniformly in
all their variables. The qualified nature of uniqueness stated in
this theorem is an indication that there will be choices to be
made at every stage of the approximation.
\end{rem}

The proof of Theorem \ref {1.1} follows the main lines of the
scheme introduced in \cite {CCO} and it is proven    in  Section
5.  There,  the complete result relating the solution of  \eqref {ch0N}  and its sharp interface limit is given.  The construction behind the proof is patterned on the Hilbert
expansion of kinetic theory. We refer the interested reader to
\cite [Subsection 3.2] {CCO} where this connection is  discussed.
We first construct an approximate solution up to order $N$ of the
chemical potential $\mu^{\e}$
 (cf. \eqref {peg1}) assuming that the left hand side of \eqref {peg1} is known and it
 is given by the Ansatz \ref {main}. This is done in Section 3.
We,  then,  insert the  constructed approximate chemical potential into \eqref {peg2}.
 The  approximate solution $m^{(N)}$ is determined  provided certain 
     compatibility conditions  are verified. This  is done in  Section 4. Finally in Section 5 we
     construct
     $ (\tilde m^{(N)}, \tilde
\mu^{(N-1)}) $      where $\tilde
m^{(N)}$ is an $\e^N$ modification of $m^{(N)}$   and $\tilde \mu^{(N-1)}$ is an
$\e^{N-1} $ modification of $\mu^{(N-1)}$ and we show Theorem \ref {1.1}.

 Let $V_j$, $  j=1,\cdots,N-1$, be the sequence
  of vector fields introduced in the Ansatz \ref {main}. According to \eqref {decp}, we split them  as
\begin{equation*}
V_j(\cdot,\G) = V^{(0)}_j(\cdot,\G) + \langle
V_j\rangle_{\eightpoint \Gamma}\ , \qquad  j=1,\cdots,N-1.
\end{equation*}
The term $V^{(0)}_j$ is determined in Theorem \ref {re.3} by
applying the Dirichelet-Neumann operator in the context of
potential theory, while the term $\langle V_j\rangle_{\eightpoint
\Gamma}$, which is constant on $\Gamma$, is determined in Theorem
\ref {re.2}.

As already explained in the introduction, the leading term $V_0$
in the vector field $\sum_{j=0}^{N-1} \e^j V_j \left(\G^{(N)}_t
\right)$ governing the interfacial flow (cf. \eqref {e0}) is
given by
$$  V_0 (\cdot,\G_t)= V^{(0)}_0(\cdot,\G_t) + \langle V_0\rangle_{\eightpoint
\Gamma_t}. $$ 
 The  family of curves   $ \Big{\{} \G_t \in \MM, t \in [0,T]\Big{\}}$ driven by $V_0$ is the solution of the
 moving boundary problem  \eqref {nov.3}.
For a curve $\G \in \MM$, the term $V^{(0)}_0(\G)$ is determined, see Lemma \ref {42} as a
real valued function on $\G$ given by
\begin{equation}
\label{ms3} V^{(0)}_0(\xi,\G) =  \left [ 
\partial_\nu \mu_{0,0,0} \right]_{\G }(\xi) \qquad \xi \in \G,
\end{equation}
where the brackets on the right denote the jump in the normal
derivative across $\G $. For each $\G \in \MM$, $\mu_{0,0,0}$ is
the solution of
\begin{equation}
\label{ms1} \Delta\mu (\xi) = 0 \qquad{\rm for}\qquad \xi
\in \Omega \setminus \G,
\end{equation}
subject to the boundary conditions
\begin{equation}
\label{ms2} \mu  (\xi) = S\left(K(\xi)  -
\frac {\int_\G  K (\xi)  \dsx }   {|\G|} \right)   +  \frac 14 \left ( B_{0,0,0}(\xi)  -
\frac {\int_\G B_{0,0,0}(\xi) \dsx }   {|\G|}   \right )
 \;\;\;\; {\rm on}\;\;\;\;\G,\;\;\;\;
   \partial_ \nu \mu  = 0\;\;\;\;{\rm
 on}\;\;\;\;
\partial\Omega.
\end{equation}
 The term  $B_{0,0,0}$ (cf. \eqref {roma.3}) is given by
  \begin{equation}
  B_{0,0,0}(\xi)=  -2  \left [  \tilde \mu_{0,0,0}(\xi)
+  G_{2,0}(\xi ) \right ] \qquad \xi \in \G,
   \end{equation}
where (cf. \eqref {roma.22})
  \begin{equation} \tilde \mu_{0,0,0}
(\xi)  =  \langle V_0\rangle_{\eightpoint \Gamma}
 2 \int_{\Gamma} G(\xi,\eta) \dse -    \int_\Om G(\xi,\eta)G_{1,0} (\eta)  {\rm
 d}\eta,
\end{equation}
for $ G(\xi, \eta) $ the Green function in $\Omega$,  with
Neumann  boundary condition on $ \partial \Omega$, satisfying the
equation
\begin{equation}
\label{2.1000}
  \Delta G(\xi, \eta)  = \delta ( \xi -
\eta) - \frac 1 {|\Omega|}  \ ,
\end{equation}
so  that
\begin{equation}
\label{NY.4}
 \int_\Om  G(\xi, \eta) d \eta = \int_\Om  G(\xi, \eta) d \xi =0\ . \end{equation}
The second step is to show that there is an actual solution of
\eqref {1a}   close to the constructed approximate solution
$m^{(N)}(\cdot, \cdot)$ whenever  both of them start from the
initial datum $m^{\eps}_0$,  see \eqref  {initial1}. This step
for the  standard Cahn-Hilliard equation (i.e. for $G_1=G_2=0$)
has been proven in the work of Alikakos, Bates and Chen, \cite
{A-B-C}, by application of spectral estimates.   We use the spectral
estimates as in \cite {A-B-C}.  Namely the liner operator
$$ \LL w = \Delta \left ( \e \Delta w - \frac 1 \e f'(m^{(N)}(t)) w\right ), $$
that one obtains linearizing \eqref {1a} at $m^{(N)}(t)$,  $t \in
[0,T]$; the solution constructed in Theorem \ref {1.1}, is the
same as in \cite {A-B-C}. The approximate solution has the
requirements  needed to apply the spectral estimates proven by
\cite {AF} in two space dimensions and by \cite {Chen2} in
arbitrary space dimensions and in more general setting. The
assumptions imposed on the forcing terms $G_1$ and $G_2$ together
with the assumption \eqref {D.10} are enough to have these terms
under control. We state the theorem and we  outline in  the appendix
the  proof.

 Let $p>0$ and $\|\cdot\|_{p,\Omega}$ be the usual norm in
$L^p(\Omega)$, then for $T>0$ we define the norm
$$\|u\|_{p,\Omega_T}:=\Big{(}\int_0^T\|u\|_{p,\Omega}^pdt\Big{)}^{1/p}.$$

 \begin{thm} \label {thFaprox}
Take $N>1$,   $G_1$ and $G_2$ as in Assumption A1.
Further assume that
\begin{equation}
\label{D.10}
 \int_\Om  G_{1,N}(\xi,\e) d \xi =0, \qquad \forall  \eps>0.  \end{equation}
Let $m^{(N)}(t)$ for $ t \in [0,T]$  where $T$ is the lifetime
of  \eqref  {e0}, be the solution of \eqref{ch0N}.   
Let   $m^\eps$ be  the solution of the generalized Cahn-Hilliard
equation \eqref{1a} supplemented by the boundary conditions \eqref
{bcgen1} and having initial datum $m^{\eps}(\xi,0)
=m^{(N)}(\xi,0)$, $\xi\in\Omega$. 
Then, there exists $\e_0>0$ so that for  all $ \e \in (0, \e_0]$,  for any pair $\lambda>\frac{13}{3}$,
$N>\frac{3\lambda+5}{3}$,  it holds that
 \begin{equation}\label{order}
\|m^\eps-m^{(N)}\|_{3,\Omega_T}\leq \eps^\lambda.
\end{equation}
  \end{thm}
 
\begin{rem}\label{onthm}
The result of the previous theorem coincides with the
analogous result in $L^p(\Omega_T)$ norm of Theorem 2.1 of
\cite{A-B-C} for the case $G_1=G_2=0$ in dimensions $n=2$, since
$3=p=2\frac{n+4}{n+2}$ and since
$\lambda>\frac{13}{3}=(n+2)\frac{n^2+6n+10}{4n+16}$.
\end{rem}
\begin {rem} \label {THF1} Set   $\tilde G_{1,N-1}= G_{1,N-1}+ \e G_{1,N} $  and therefore $\tilde G_{1,N}=0$.
Determine  the velocity field $V_{N-1} $  replacing $ G_{1,N-1}$
with $\tilde  G_{1,N-1}$. In this way   the   condition \eqref
{D.10} of Theorem \ref {thFaprox}  is  trivially satisfied.
\end {rem}
 \section {The  Construction of the Approximate Chemical
potential}\label{sec3} In this section, we apply
potential theory to show  the following result.
    \begin {thm}  \label {re.2}
 Take $N>1$ and  $G_1$  as in Assumption A1.
 Let $\G^{(N)}_t$, $t \in [0,T]$, be the solution of \eqref {e0} in
${\mathcal  M}$,  with $T$ its lifetime, see \eqref {ll.2}.  Let  $m^{(N)}(\cdot,
\G^{(N)}_t)$  be  as  in  the Ansatz \ref{main}.
 There is a
 unique way to determine the
$\langle V_j\rangle  (\G^{(N)}_t) $, $j=0,\cdots,(N-1)$, such
that there exists a unique (up to a constant in $\xi$) expansion
   \begin{equation}
\label{FF.1} \mu^{(N-1)}(\xi, t)=
 \sum_{i=0}^{N-1} \e ^i
\mu_i (\xi,t)    \qquad  \hbox { in } \Omega \times (0,T),
\end{equation}
 with
\begin{equation}
\label{L.1b}
   \partial_t  m^{(N)}(\xi,t)  =
   \Delta
\mu^{(N-1)}(\xi, t) + \sum_{j=0}^{N-1}  \e^j G_{1,j}(\xi) +
R_1(\xi, t, \e) \qquad  \hbox { in } \Omega \times (0,T),
\end{equation} with $ R_1$ given in \eqref{remainder}. Further
$\mu^{(N-1)}  (\cdot, t)$, for  $ t \in  (0,T)$, is a $C^{\infty}
(\Om)$  function satisfying the  Neumann  homogeneous  boundary
conditions on $
\partial \Om$,
  \begin{equation}
\label{2.014}
 \sup_{{\xi, t} \in \Om \times [0,T]}  | R_1(\xi, t, \e) | \le
 C(T)\e^{N-1},
\end{equation} and
\begin{equation}
\label{2.015} \sup_{ t \in [0,T]} | \int_\Om  R_1(\xi, t, \e)   d
\xi | \le C(T) \e^N ,
\end{equation}
 where $C(T)$ is a constant independent of $\e$.  Moreover, the
 terms
$\mu_i $ appearing in \eqref{FF.1} are specified by
\eqref{peg1as1}, \eqref{rep1} and \eqref{rep10} below.
\end {thm}

  We look for a function $\mu^{(N-1)} $  from $ \MM$ to $ C^\infty (\Om)$ having  the form
\begin{equation}
\label{2.09} \mu^{(N-1)} (\xi,\G)  = \sum_{i=0}^{N-1} \e^i\mu_i(
\xi, \G)  \qquad   \xi \in \Om , \qquad    \partial _\nu \mu_i  = 0  \quad \hbox {on} \quad \partial \Om,
\end{equation}
where   $\mu_i$, $ i=0,\cdots, N-1$ are functions to be
determined.
  We insert into \eqref{peg1} the function $ m^{(N)}$, given by  the Ansatz \ref{main},   and
$\mu^{(N-1)}$ given by \eqref{2.09}, where both are evaluated  at
$\Gamma:=\G^{(N)}_t$, for  $\G^{(N)}_t$ the solution of
\eqref{e0}. Therefore, we obtain $(N-1)$ Laplace equations  for
$\mu_i(\cdot, \G^{(N)}_t) $, $i=1,\cdots,(N-1)$. The compatibility
conditions are needed in order to solve  these equations and
determine $\langle V_j \rangle (\G^{(N)}_t)$  for
$j=0,\cdots,(N-1)$.

\vskip0.5cm\noindent

When differentiating $m^{(N)} (\cdot, \G^{(N)}_t) $ with respect
to $t$ we  need to  take into  account that $m^{(N)}$ depends on
$\G_t$ through a fast and slow scale.  The fast scale brings a
factor $\e^{-1}$.

\begin{defin}
Let $m$ be a function from ${\mathcal M}$ to  $C^\infty(\Om)$  of
the type $ h  \Big{( }\frac { d(\xi,\G)} \e, s(\xi,\G) \Big{)}$
and $V$ be a vector field on   $\MM$. We define
\begin{equation}
\label{gi.3} D_V m (\xi,\G) := \frac 1 \e   h' \Big{(} \frac {
d(\xi,\G)} \e, s(\xi,\G)  \Big{)} V (s(\xi)),
\end{equation} where the prime indicates the derivative of $h$ with respect to the first variable $
z=\frac { d(\xi,\G)} \e$.

In addition, for any $W_N := \sum_{j=0 } ^{N-1} \e^j V_j $ with
$V_0,\cdots,V_{N-1}$ vector fields on $\MM$, we define
\begin{equation}\label{e3}
D_{W_N} m (\xi,\G) :=  \sum_{j=0 }^{N-1} \e^j D_{V_j} m (\xi,\G).
\end{equation}
\end{defin}

 Note that by the orthogonality of $\nabla_{\xi} d$
with respect to the surface there
 is no contribution  in \eqref {gi.3} from
$s(\xi,\G)$. Therefore, cf. \cite{CCO} for the detailed
computations, differentiating \eqref{e1} with respect to $t$,
applying the chain rule at the right-hand side (here the velocity
will appear since $m_j$ are defined on $\Gamma$) and then using
\eqref{e2}, \eqref{gi.3} and \eqref{e3}, we arrive at
\begin{equation}
\label{flux1} \begin{split}    \partial_t   m^{(N) }  =&
D_{W_N}(m^{(N)})  = D_{V_0}m_0
+ \e \left[ D_{V_1}m_0 + D_{V_0}m_1\right]\\
&+ \e^2 \left[ D_{V_1}m_1 + D_{V_0}m_2 +   D_{V_2}m_0\right]
+\cdots+
    \e^{N-1} \left [ \sum_{i=0}^{N-1} D_{V_i}m_{N-1-i}  \right] + R_{N} +
    E,  \end{split}
  \end{equation}
where
\begin{equation}
\label{flux20} R_{N}    \equiv  \e^{N}    \left [ \sum_{i=0}^{N-1}
D_{V_i}m_{N-i} \right] + {\mathcal O}(\e^{N}).   \end{equation}
The $m_0$ is the function defined in \eqref {e1a} and the term
$E\equiv  E(\xi, t, \e)$ is  obtained by differentiating
${\displaystyle  r \Big{(}\frac {d(\xi,\G_t)} {\e_0 }\Big{)} }$,
the unimodal function appearing in the definition of $m_0$, with
respect to the velocity field. $E$ is given by
 \begin{equation}
\label{PA.2}  E = \frac 1{\e_0 } r'\Big{(}\frac {d(\xi,\G_t)}
 {\e_0}\Big{)}  \left [ \sum_{i=0}^{N-1} \e^i V_i(\s(\xi),t) \right ]
\Big{ \{}   \bar m-  \Big{ [} \1_{\{d(\xi, \G_t)>0 \}} -
\1_{\{d(\xi, \G_t)<0 \}}  \Big{ ]}    \Big{ \}}.
\end{equation}
  Note that $E$ is exponentially small since $r'$ is
different from zero only for $ \frac {\e_0}{2 \e} \le |z| \le
\frac { \e_0 } {\e} $ while $\bar m$ converges exponentially fast
to $\pm 1$ as $z\rightarrow \pm \infty$, \cite{CCO}. Taking into
account \eqref{flux1} and
  \eqref{peg1} we    obtain  a set of $N$ equations for the $\mu_i$,
$i=0,\cdots, N-1$.

\vskip0.5cm  \noindent {\bf Zero order term    in $ \e$: }

\begin{equation}
\label{2.010}  \left \{ \begin {split} &
 D_{V_0}m_0 \equiv \frac 1  \e V_0   m'=
 \Delta
\mu_{0}   + G_{1,0}   \qquad \hbox {for}\qquad  \xi \in \Om,  \quad   t \in [0,T],
\cr &      \partial_ \nu \mu_0 = 0  \quad \hbox {on} \quad \partial \Om,   \end {split} \right .
      \end{equation}
where $G_{1,0}$ is the zero order term in $\e$ of $G_1$.

\vskip0.5cm \noindent {\bf  First order term    in $ \e$: }

\begin{equation}
\label{2.011}    \left \{ \begin {split} & \Big{[ }D_{V_1}m_0 +
D_{V_0}m_1\Big{]}
   =
\Delta \mu_1  + G_{1,1} \qquad \hbox {for} \qquad  \xi \in \Om, \quad   t \in [0,T], \cr &
  \partial_\nu \mu_1 = 0  \quad \hbox {on} \quad \partial \Om.
  \end {split} \right .
 \end{equation}

\vskip0.5cm  \noindent {\bf n-th order term    in $ \e$  ($n \le
N-1$): }

\begin{equation}
\label{2.112}      \left \{ \begin {split} &  \left[ \sum_{i=0}^n D_{V_i}m_{n-i}  \right]
     =
  \Delta
\mu_n  + G_{1,n} \qquad \hbox {for} \qquad  \xi \in \Om, \quad   t \in [0,T],
\cr &
 \partial _ \nu \mu_n = 0  \quad \hbox {on} \quad \partial \Om.
  \end {split} \right .
   \end{equation}

\vskip0.5cm \noindent  {\bf Remainder  term:} The remainder term,
see \eqref{flux20}  and \eqref{PA.2}, is  given by
\begin{equation}
\label{remainder} R_1( \xi, t, \e) =   \e^N G_{1,N} (\xi,\e) +
R_{N} (\xi,t,\e)+    E(\xi, t, \e).    \end{equation} Since the
derivative in $R_N$ (cf. \eqref {flux20}) brings down a factor
$\e^{-1}$ then this yields easily the next estimate for $R_1$
\begin{equation}
\label{Ap.1} \sup_{ (\xi,t) \in \Om \times [0,T]} | R_1 (\xi,t)|
\le C(T) \e^{N-1} . \end{equation}
 Further, one   gains  an extra power of
$\e$  when integrating  $R_1$, since the terms  of order
$\e^{N-1}$   have support in  $\NN(\e_0)$
\begin{equation}
\label{Ap.2}
 \sup_{  t \in   [0,T]} \int_{\Om}  | R_1 (\xi,t,\e)| {\rm d } \xi  \le C(T) \e^N.  \end{equation}

In the sequel, we prove existence and uniqueness (up to a
constant) of solutions of the equations obtained so far at
different orders. In Lemma \ref{230} and in Lemma \ref{231} we
consider the first and second order equation respectively.
Finally, in Lemma \ref{234} we outline the proof for solving the
equation to a generic order. In the next lemma we write in an
explicit way the dependence of the mean velocity on $\e$. This is
done in order to get easily the leading velocity field governing
the interfacial flows, see \eqref {e0}.

 \begin{lem} \label{230}
 Under the conditions
 \begin{equation}
\label{zcond1} \begin {split}  &
 V_0 (\cdot,\G_t)= V^{(0)}_0(\cdot,\G_t) + \langle V_0\rangle_{\eightpoint
\Gamma_t}[1+    c_0(\e)],  \cr &  \int_{\G_t}V^{(0)}_0(\eta
,\G_t) \dse  =0, \quad \langle V_0\rangle_{\eightpoint \Gamma_t}
 =   \frac 1 {2 |\G_t|}  \int_\Om  G_{1,0} (\eta)  {\rm d}\eta\;\;\;\;  t \in [0,T] ,  \end {split}
   \end{equation}
where $c_0(\e)$ defined in \eqref {ce1} goes to zero
exponentially fast  as $\e \to 0$,
 there exists a unique solution (up to constants in $\xi$) of \eqref{2.010} given by
\begin{equation}
\label{peg1as1} \mu_0 (\xi,\G_t) = \mu_{0,0}(\xi,\G_t) + \tilde
\mu_0 (\xi,\G_t),  \end{equation} where
\begin{equation}
\label{j1}  \mu_{0,0}(\xi,\G_t) =
 \int_\Om G(\xi,\eta)\left(
{1\over \e}m_0'\left({d(\eta,\G_t)\over
\e}\right)V^{(0)}_0(s(\eta),\G_t)\right) {\rm d}\eta   +   c_0(t).
\end{equation} Here, $c_0(t)$ is a constant (in $\xi$) to be
determined, and
\begin{equation}
\label{j2} \tilde \mu_0 (\xi,\G_t)  =  \langle
V_0\rangle_{\eightpoint \Gamma_t} [1+    c_0(\e)]
 \int_\Om G(\xi,\eta)\left(
{1\over \e}m_0'\left({d(\eta,\G_t)\over \e}\right) \right) {\rm
d}\eta -    \int_\Om G(\xi,\eta)G_{1,0} (\eta)  {\rm d}\eta.
\end{equation}
The term $\mu_0$ is in $C^\infty (\Om)$ for any $ t \in [0,T]$.
\end{lem}
 \begin{proof} Because  $ \partial_ \nu\mu_0  = 0$   on  $\partial \Om$,
 the solvability of \eqref{2.010} requires that for all $t \in [0,T]$
\begin{equation}
\label{exact1}
 \int_\Om \left(
{1\over \e}m_0'\left({d(\eta,\G_t)\over
\e}\right)V_0(s(\eta),t)\right) {\rm d}\eta  - \int_\Om  G_{1,0}
(\eta)  {\rm d}\eta
   =0 .    \end{equation}
In two dimensions (note that in three dimensions there will be
extra terms), by using local coordinates it follows that
  \begin {equation}  \label {rm.5}  \begin  {split}
  & \int_\Om \left(
{1\over \e}m_0'\left({d(\eta,\G_t)\over
\e}\right)V_0(s(\eta),t)\right) {\rm d}\eta = \int_{\NN(\e_0)}
\left( \frac 1 \e m_0'\left({d(\eta,\G_t)\over
\e}\right)V_0(s(\eta),t)\right) {\rm d}\eta  \cr & = \frac 1 \e
\int_{\G} \int_{-\frac   {\e_0} \e}^{\frac  {\e_0} \e} m_0'(z)
V_0(s,t)  \e(1-\e zK(s)){\rm d}s{\rm d}z\cr &= \int_{\G}
\int_{-\frac  {\e_0}\e}^{\frac  {\e_0} \e}  m_0'(z) V_0(s,t) {\rm
d}s{\rm d}z - \e \int_{\G} \int_{-\frac {\e_0} \e}^{\frac {\e_0}
\e}  z m_0'(z) K(s) V_0(s,t) {\rm d}s{\rm d}z \cr  &
 =  \int_{\G} \int_{-\frac  {\e_0} \e}^{\frac  {\e_0} \e}  m_0'(z) V_0(s,t) {\rm d}s{\rm d}z  =
 2(1- e^{-\frac  {\e_0} \e})
\int_{\G_t}V_0(\eta ,\G_t) \dse.  \end {split}
\end{equation}
The last line holds true since $m_0'$ is even and exponentially
decreasing. Replacing  now \eqref  {rm.5} in \eqref{exact1}, we
obtain
\begin{equation}\label{help1} \int_\Om
G_{1,0} (\eta)  {\rm d}\eta=    2 (1- e^{-\frac  {\e_0} \e})
\int_{\G_t}V_0(\eta ,\G_t) \dse.
\end{equation}
 Taking into account the splitting  $V_0= V^{(0)}_0 +\langle V_0\rangle_{\eightpoint \Gamma}[1+    c_0(\e)] $,
 (cf. the first and second equality in \eqref{zcond1}),
and using \eqref{help1}, we arrive at
$$
 \int_{\Gamma_t}V_0 (\eta,\G_t)dS_\eta =  |\Gamma_t|\langle V_0\rangle_{\eightpoint \Gamma} [1+    c_0(\e)] =    \frac{1}{2}  (1+    c_0(\e)) \int_\Om
G_{1,0} (\eta)  {\rm d}\eta,
$$
where
\begin {equation}  \label {ce1} c_0(\e):= \frac { e^{-\frac  {\e_0} \e}} {   (1- e^{-\frac  {\e_0} \e})}.
\end {equation}
This forces to take $V_0$ satisfying the third relation of
\eqref{zcond1}. By potential theory, once the compatibility
condition  is satisfied, the solution is given by \eqref
{peg1as1}.
\end{proof}
\begin {rem} Note that  $  \langle V_0\rangle_{\eightpoint
\Gamma_t} $ and  $ \tilde \mu_0 (\xi,\G_t) $ are  completely
determined once we know $m_0$. The quantity $
\mu_{0,0}(\xi,\G_t)$  depends on $c_0(t)$ and $V^{(0)}_0$.  These
quantities  will be  determined when proving Theorem \ref {re.3}.
\end {rem}
\subsection{The first order term in $ \e$ } For the derivation
of the first order correction we need to prove the solvability of
equation \eqref {2.011}.

 \begin{lem}\label {231}   There exists a unique (up to constants in $\xi$) solution $\mu_1$ of
\eqref{2.011}
 provided that
\begin{equation}
\label{ov.2}
V_1(\G_t) \equiv   V_1^{(0)}(\G_t) + \langle
V_1\rangle_{\eightpoint \G_t },  \end{equation} with
\begin{equation}
\label{exact2} \int_{\G_t} V^{(0)}_1(\eta, \G_t)\dse =   0,     \qquad
\forall t \in [0,T],  \end{equation} and $\langle
V_1\rangle_{\eightpoint \G_{t}}$  chosen  according to
$$ \langle V_1\rangle_{\eightpoint \G_t } =
\frac 1{ 2 |\G_t|} [1+c_0(\e)] \left [  \int_{\Om} G_{1,1}(\xi)d\xi    -  b_1(t) \right ], $$ where $b_1(t)$ is defined in
\eqref{svm2} and $c_0(\e)$ in \eqref {ce1}. The solution  is
given  by
\begin{equation}
\label{rep1} \mu_1 (\xi,t) = \mu_{1,0}(\xi,t)+  \tilde
\mu_{1}(\xi,t),  \end{equation} where  $\tilde \mu_{1}$ is defined
in \eqref{div3}, while
\begin{equation}
\label{rep2}
 \mu_{1,0}(\xi,t) = \int_\Om G(\xi,\eta)\left(
{1\over \e}m_0'\left({d(\eta,\G_t)\over
\e}\right)V^{(0)}_1(s(\eta),t)\right) {\rm d}\eta  +   c_1(t).
\end{equation}
Here, $c_1(t) $ is a constant (in $\xi$) to be determined. In
addition, the solution  is a $C^\infty (\Om ) $ function for any
$t\in [0,T]$.
\end{lem}
 \begin{proof}
The solvability of \eqref{2.011} requires that
\begin{equation}
\label{cond1} \int_\Om \Big{ [}     D_{V_1}m_0 + D_{V_0}m_1 \Big{
]} d\xi   -  \int_\Om G_{1,1}(\xi)d\xi
 =0,
\end{equation} for any $ t \in [0,T]$.
Here, we are assuming that   $m_1$, $m_0$  and  $V_0$ are  already
determined   and so we define
\begin{equation}
\label{svm2} b_1(t)  = \int_\Om   D_{V_0}m_1    {\rm d}\xi  .
\end{equation} Proceeding as in \eqref
{rm.5}     we obtain
\begin {equation}
\label{svm4}  \int_\Om D_{V_1}m_0 {\rm d}\xi    = 2  (1- e^{-
\frac  {\e_0} \e } )   \int_{\G_t} V_1(\eta)\dse .
 \end {equation}
Taking into account the decomposition \eqref{ov.2}  and relation
\eqref{exact2}  we have
$$        \int_{\G_t} V_1(\eta, \G_t)\dse=  |\G_t| \langle
V_1\rangle_{\eightpoint \G_t }.  $$ Therefore, relation
  \eqref{cond1} is satisfied  if

\begin{equation}
\label{svm5} \langle V_1\rangle_{\eightpoint \G_t }  =   {1\over
2 |\G_t|} [1+c_0(\e)] \left [   \int_\Om G_{1,1}(\xi)d\xi
-b_1(t)\right ],
\end{equation}
where $c_0(\e)$ is defined in \eqref {ce1}. This determines
$\langle V_1\rangle_{\eightpoint \G_t }$, the projection of
$V_1(\G_t)$ onto the constants. The solution  of \eqref{2.011}
exists  and  it is given by
 \begin{equation}
\label{C.2}
 \begin{split} & \mu_1 (\xi, t ) = \int_\Om G(\xi,\eta)
   \Big{[} D_{V_1}m_0 + D_{V_0}m_1 \Big{]}
  {\rm d}\eta   -   \int_\Om G(\xi,\eta)  G_{1,1}(\eta) {\rm d}\eta+
 c_1(t).  \end{split}  \end{equation}
Since we shall use the  decomposition \eqref{ov.2},  it is
convenient to write
\begin{equation}
\label{div} \mu_1 (\xi,t) = \mu_{1,0}(\xi,t)+  \tilde
\mu_{1}(\xi,t),   \end{equation} where $\mu_{1,0}(\xi,t)$ is
given in \eqref{rep2} and $\tilde\mu_1$ is defined as follows
\begin{equation}
\label{div3}
  \tilde \mu_1 (\xi,t)=  \int_\Om G(\xi,\eta)
     D_{V_0}m_1
  {\rm d}\eta     +  \langle V_1\rangle_{\eightpoint \G_t } \int_\Om G(\xi,\eta)\left(
{1\over \e}m_0'\left({d(\eta,\G_t)\over \e}\right) \right) {\rm
d}\eta     - \int_{\Om} G(\xi,\eta) G_{1,1}(\eta) {\rm d}\eta.
   \end{equation}
\end{proof}

\begin{lem}\label {234}  The solution $\mu_j (\cdot,t)$ of \eqref {2.112}
for $2\le j \le N-1$ exists  and is unique (up to constant in
$\xi$) provided that
\begin{equation}  \label{svm41}
V_j(\G_t)    \equiv V_j^{(0)}(\G_t) + \langle
V_j\rangle_{\eightpoint \G_t } ,
\end{equation}
\begin{equation}
\label{svm40} \int_\G V_j^{(0)}(s,\G_t){\rm d}s = 0, \qquad \forall  t \in
[0,T],
\end{equation}
 and $\langle V_j\rangle_{\eightpoint \G_t }$
is chosen according to \eqref{svm50}. It is given by
\begin{equation}
\label{rep10} \mu_j (\xi,t) = \mu_{j,0}(\xi,t)+  \tilde
\mu_{j}(\xi,t),   \end{equation} where
\begin{equation}
\label{rep22}
   \mu_{j,0}(\xi,t) = \int_{\Om} G(\xi,\eta)\left(
\frac 1 \e m_0'\left({d(\eta,\G_t)\over
\e}\right)V^{(0)}_j(s(\eta),t)\right) {\rm d}\eta + c_j(t),
\end{equation} and
\begin{equation}
\label{div4}
 \begin{split}  \tilde \mu_j (\xi,t)=&  \int_{\Om} G(\xi,\eta)    \left[ \sum_{n=0}^{j-1} D_{V_n}m_{j-n}
\right]
       {\rm d}\eta  \\ &  +  \langle V_j\rangle_{\eightpoint \G_t } \int_{\Om} G(\xi,\eta)\left(
\frac 1 \e m_0'\left(\frac {d(\eta,\G_t)} \e \right) \right) {\rm
d}\eta    +     \int_\Om G(\xi,\eta)  G_{1,j}(\eta) {\rm d}\eta.
   \end{split} \end{equation}
The solution $\mu_j (\cdot,t)$ for $t \in (0,T]$ is a $C^\infty
(\Om) $ function.   \end{lem}
 \begin{proof} The proof is analogous to the proof of Lemma \ref{231}.  The solution  exists if
\begin{equation}
\label{rm.1}
  \int_{\Om} \left[ \sum_{n=0}^j  D_{V_n}m_{j-n}  \right]  {\rm d}\xi -     \int_\Om G_{1,j}(\xi)d\xi
=0,   \end{equation} for any $ t \in [0,T]$. Here, $D_{V_n}m_{j-n}
$  for $n=0,\cdots,j-1$ are determined  and so, we define
\begin{equation}
\label{svm30} b_j(t)  =
  \int_{\Om} \left[ \sum_{n=0}^{j-1}  D_{V_n}m_{j-n}  \right]  {\rm d}\xi. \end{equation}
Requiring  \eqref {svm41}  and
  \eqref {svm40}
 we obtain
\begin{equation}
\label{svm45}   \int_{\Om} D_{V_j}m_0 {\rm d}\xi = 2|\G_t|  [1-
e^{- \frac {e_0} \e } ] \langle V_j\rangle_{\eightpoint \G_t }.
\end{equation} Hence, to fulfill relation \eqref{rm.1}, we must take
\begin{equation}
\label{svm50} \langle V_j\rangle_{\eightpoint \G_t }  =   \frac
1{ 2 |\G_t|} [1+c_0(\e)]  \left [  \int_\Om G_{1,j}(\xi)d\xi-
b_j(t)\right ],  \end{equation} where $c_0(\e)$ is defined in
\eqref {ce1}. This determines $\langle V_j\rangle_{\eightpoint
\G_t }$, the projection of $V_j(\G_t)$ onto the constants. It
still  remains to determine the orthogonal part $V_j^{(0)}$. The
solution of \eqref{2.112} exists and is given by \eqref{rep10}.
 \end{proof}
\subsection{Proof of Theorem \ref{re.2}}\label{subsecc3}

\noindent From Lemma \ref{230}, Lemma \ref{231} and Lemma
\ref{234}  we have that $  \mu^{(N-1)} $ satisfies by
construction \eqref{L.1b}. The remainder $R_1$ is defined in
\eqref{remainder} and estimated in \eqref{Ap.1} and \eqref{Ap.2}.
The term $\mu^{(N-1)} (\cdot, t)$ for $ t\in [0,T]$ satisfies the
homogeneous Neumann  boundary conditions by construction. Thus,
Theorem \ref{re.2} holds  true. \qed

\section {Derivation of the equations for $m^{(N)}$}
 Next
theorem assures the existence and (essential) uniqueness of the
functions $m_j$, $j=0,\cdots,N$, having the properties required in
the Ansatz \ref {main}.  Existence and uniqueness are obtained
provided that a certain compatibility condition is satisfied. This
determines  $ V^{(0)}_j$, the orthogonal part of the velocity
field. \vskip0.5cm
\begin {thm} \label{re.3}  Take $N>1$ and $G_2$ as in Assumption A1.
Let $T$ be the lifetime of  the solution of \eqref{e0} in
${\mathcal  M}$.  Let  $\mu^{(N-1)} (\cdot, t) $,  $t \in [0,T]$
be the function constructed   in Theorem
 \ref{re.2}. Then it is possible to choose the vector fields $V^{(0)}_j$ so that
  there exist
$m_j $, $j=0,\cdots,N$, having the properties prescribed in  the
Ansatz  \ref {main} such that
\begin{equation}
\label{G.1} \mu^{(N-1)}(\xi,t)  = -\e \Delta   m^{(N)}(\xi,t)
   +\frac 1 \e f(m^{(N)}(\xi,t))  - \sum_{j=0}^{N-1}  \e^j G_{2,j}(\xi) +  R_2(\xi,t, \e)   \qquad
\hbox {{\rm in}}\qquad
 \Omega \times (0,T],   \end{equation}
with $R_2$ given by \eqref{v.100}.
 Further, $m^{(N)}(\cdot,t) $ for $ t\in [0,T]$ is a $C^{\infty} (\Om)$ function that satisfies
the homogeneous Neumann  boundary conditions and
\begin{equation}
\label{NY.1} \sup_{\xi \in \Om} \sup_{t \in [0,T]} |R_2(\xi,t,
\e) | \le C \e^{N}.  \end{equation} Finally, the choice of the
term $V^{(0)}_j$ is specified by the equations
 \eqref{z.1}, \eqref{z.1A} and \eqref{zz.1} given below.
\end {thm}
\vskip0.5cm

In \eqref{peg2} we insert  at the left-hand side  the already
determined function $\mu^{(N-1)} (\cdot, \G^{(N)}_t)$ and we
obtain
\begin{equation}
\label{G.2} \mu^{(N-1)}(\xi, t)  = -\e \Delta     m (\xi,t)
   +  \frac{f( m  (\xi,t) )}{\eps}     -   G_2(\xi;\e)  \qquad
\hbox {{\rm in}}\qquad
 \Omega \times (0,T).    \end{equation}
Then $ \mu^{(N-1)}$  is  written  in terms of $ m^{(N)} $, chosen
according  to  the ansatz. Here, we prove  that  there exists a
unique way to find the function $ m^{(N)} $, having indeed the
property required in the ansatz and satisfying equation
\eqref{G.2} in a certain sense (to be specified in the sequel).

  The existence of  any $ m_j$, $j=0,\cdots,N$  is
obtained  provided  that a  compatibility  condition is
satisfied. This  compatibility  condition forces us to define
properly $V^{(0)}_j$, $j=0,\cdots,(N-1)$.

We distinguish two main steps:\\
 \noindent {\bf $\bullet$ Step
1:}
 \noindent  {\it Determination at any order  of the  equations. This is carried out in  the   Subsection \ref{suba}.}

\noindent {\bf $\bullet$  Step 2:}
 \noindent  {\it Analysis of  the equations derived   at  step 1.
 This will be done in the Subsection \ref {suba2}.}

In the present and in the next subsection $ \G_t$ is kept fixed,
 therefore, for the sake of a simpler notation we drop the subscript $t$, except
of the cases that a subscript use may add some clarity in our
arguments.

\subsection { Determination  of the  equations for $m_j$, $j=0, \cdots N$.
}\label{suba} To separate the fast and slow scale  of   $m^{(N)}$
near the surface $\G$, we  write the Laplacian  in the system of
local coordinates introduced in  Subsection 2.1.  The expansion
in $\e$ of the  Laplacian   written  in this coordinate system is
reported in  the Appendix, subsection  A.2.  We then match the
right and left terms of the equations
  having the same  power of  $\e$,  distinguishing  the case where  $     \xi
\in  \NN (\e_0)$ from the  one  with $     \xi \in \Omega\setminus
\NN (\e_0)$. We therefore, get  at any  order  two sets of
equations: one for  $     \xi \in  \NN (\e_0)$ and the other for
$     \xi \in \Omega\setminus  \NN (\e_0)$.  Since the interface
separates $\Om$ in  two regions we will distinguish those $ \xi
\in \Omega\setminus  \NN (\e_0)$ which are inside $\G$ from those
$     \xi \in \Omega\setminus  \NN (\e_0)$ which are outside
$\Gamma$.

Taking into account   formula \eqref{C.100} in  the Appendix,
denoting by $'$ the derivative with respect to $z$,  and by
$a_n$, $b_n$, $c_n$   the quantities defined in \eqref{coef},
after simple, however lengthy computations we obtain the
following  identity
\begin{equation}
\label{C.20}
 \begin{split}   \e^2 \Delta m^{(N)} (z,s)
= &\left \{
  \bar m''(z) +
  \sum_{n=1}^N \e^n  \left [   h''_n(z,s) +a_n (z,s)  \bar m' \right ] \right \}
 \\ & +
 \left \{
  \sum_{n=2}^N \e^n     \sum_{i=1}^{n-1}   a_{n-i}(z,s)
  h'_i (z,s)  + \sum_{n=3}^N \e^n \left [
\sum_{i=1}^{n-2}  b_{n-i}(z,s) \frac {d^2}{d s^2} h_i(z,s) \right
] \right.\\ &+  \left. \sum_{n=4}^N \e^n
  \sum_{i=1}^{n-3}c_{n-i}(z,s) \frac {d }{d s } h_i(z,s)   \right \}
+
 \e^2 \Delta \left [ \sum_{i=1}^N  \e^i
\phi_i(\xi)     \right ]  \cr & +E_1 (\xi,t, \e) + \e^{N+1}  A
(\xi,t, \e),  \end{split}   \end{equation} with
\begin{equation}
\label{ap.100}
 \sup_{(\xi,t) \in \Omega \times [0,T]} | A(\xi,t,
\e)| \le C(T),  \end{equation}
\begin{equation}
\label{ap.101}
 \sup_{ t  \in [0,T]} \int_\Omega d\xi| A(\xi,t,
\e)| \le  \e C(T),  \end{equation}

\begin{equation}
\label{no.1}
 \begin{split} &  E_1 (\xi,\e) \equiv
\e^2 \Delta  r (\frac {d(\xi,\G)} {\e_0}) \left \{    \bar m(\frac
{d(\xi, \G)} {\e })  -  \left  [ \1_{\{d(\xi, \G)>0 \}} -
\1_{\{d(\xi, \G)<0 \} }  \right ]   \right \}
 + 2 \e^2 \nabla r \cdot \nabla
\left [ \bar m(\frac {d(\xi, \G)} {\e })  \right]  \end{split},
\end{equation}
and \begin{equation} \label{st.2}
 \lim_{\e\to 0} \sup_{(\xi,t) \in \Om \times [0,T]} |E_1 (\xi,t,\e)| = 0.  \end{equation}

Let us now define $f_i$ such that
\begin{equation}
\label{C.3}
 \begin{split} &f(m^{(N)}) = f(m_0) + f'(m_0)  \left [\sum_{i=1}^N
 \e^i  m_i  \right ] +  \sum_{i=2}^N  \e^i f_i (m_0,m_1,..,m_{i-1})
 +   \e^{N+1}  B_{N+1}( \cdot , \e),  \end{split}  \end{equation}
\begin{equation}
\label{st.1}
 \sup_{ \xi \in \Omega, t \in [0,T]} |  B_{N+1}( \xi, t, \e)| \le C.  \end{equation}
One can easily obtain $f_i$ for any $i=2,\cdots, N$ by using
Taylor expansions up to $N-1$ order  for $f $  around $m_0$ and
collecting then the terms of the same power of $\e$.

We insert   \eqref{C.20} and  \eqref{C.3} into \eqref{G.2} and
equate terms having the same order in $\e$ (when estimated with
the $L^\infty (\Om)$ norm) obtaining, this way, two equations at
any order,  one for $ \xi \in \Omega \setminus \NN (\e_0)$, and
the other for $ \xi \in
  \NN (\e_0)$.  The  equation for $
\xi \in \Om \setminus \NN (\e_0)$ will determine $\phi_i$ which
are the slowly varying terms,  while the one   for $ \xi \in
  \NN (\e_0)$ will determine  $h_i$ i.e. the rapidly decaying terms.
When  deriving the equations  for   $ \xi \in \NN (\e_0)$ then
terms of the type $  \Delta \phi_i( \e z, s) $ appear.  The $
\phi_i$
 are
$C^\infty$ functions, since they are  proportional to $\mu_i$,
and they
 have the same type of singularity in $\e$ when
differentiated in $\xi$. So, the terms    $ \e^{n+1} \Delta
\phi_{n-1}( \e z, s)$  are of order  $  {\mathcal O}(\e^n)$, and
therefore, we write them  in the $\e^n$ order equation.

In  the following we will use the notation $ f (\pm 1)$  (or
$f(m(\pm \infty)))$   in $\Om \setminus \NN (\e_0)$. This refers
to the fact that $ \Gamma$ separates $\Om$ in two sets, i.e. $
\Omega^+$ which is the part outside $\Gamma$ and $ \Omega^-$ the
part inside $\Gamma$. Therefore, $\pm$ refers to different regions
of $(\Om \setminus \NN (\e_0)) \cap \Om^\pm $. It is convenient
to  write the external potential $G_{2, i}$, for $i=0, \cdots, N$
in local coordinates  when $ \xi \in \NN (\e_0)$.  We therefore
identify $G_{2, i}(\xi) = G_{2, i}(d(\xi,\G),s(\xi,\G))$,  for $i=0,
\cdots, N$.
 \vskip0.5cm

\noindent
 { \bf Zero order term in $ \e$: }
 Matching gives
\begin{equation}
\label{2.221} 0 = -  r \Big{(}\frac \e  {\e_0} z\Big{)}   m_0''(z)
+ f(m_0 (z))\;\;\;\;\hbox {for}\;\;\;\;z \in  \Big{[}-\frac {\e_0}
{  \e},\frac {\e_0}{  \e} \Big{]},
\end{equation}
and
\begin{equation}
\label{M.1}
 f(m_0 (\pm \infty) ) =0\;\;\;\;
\hbox {for}\;\;\;\;  \xi \in \Om \setminus \NN (\e_0).
\end{equation}
 
\vskip0.5cm

\noindent
 { \bf  First  order term   in $ \e$: }
Again matching gives for $z \in  \left [-\frac {\e_0}
{  \e},\frac {\e_0}{  \e} \right ]$ and $s \in \G$
\begin{equation}
\label{2.20}
 \begin{split} & \mu_{0}( \e z, s) = -
\left [   h''_1 (z,s) - K(s)m_0'(z) \right ] +f'( m_0)
 \left [ h_1(z, s )   +  \phi_1(\e z,
s)   \right ]  - G_{2,0} ( \e z,s)   \end{split}     \end{equation}
and
 \begin{equation}
\label{2.200}
  \mu_{0}( \xi) = f'( \bar m ( \pm \infty )) \phi_1(\xi)  - G_{2,0} ( \xi)
\;\;\;\;\hbox {for}\;\;\;\;\xi \in \Omega \setminus \NN (\e_0).
\end{equation}

\vskip0.5cm

\noindent
 { \bf Second   order term   in $ \e$: }
\begin{equation}
\label{2.m8}
 \begin{split}
\mu_1 (  \e z, s)     = & - \left [  h''_2(z,s)- K^2(s) z
m_0'(z)- K (s)
 h'_1(z,s ) \right] \\ & +f'(m_0(z)) \left [ h_2(z,s)   +   \phi_2(
  \e z, s)   \right ]  -\e  \Delta \phi_1( \e z, s)   + f_2(m_0,m_1)(\e z,s)  - G_{2,1} (\e z,s),   \end{split}
\end{equation}

\begin{equation}
\label{2.c8}
 \mu_1 ( \xi )  = f'(\bar m ( \pm \infty))
\phi_2( \xi)  + f_2(\bar m ( \pm \infty),\phi_1  (\xi)) -\e
\Delta \phi_1(\xi) - G_{2,1} (\xi)\;\;\;\;\hbox {for}\;\;\;\; \xi
\in \Om \setminus \NN (\e_0).    \end{equation}

More explicitly the $f_2$ term for $\xi \in  \NN (\e_0)$ is given
by
\[
   f_2(m_0,m_1) (\e z, s)  =
\frac 12  f''(m_0 (z))  \left [ h_1^2 (z,s ) + \phi_1^2( \e z, s)
+
   2\phi_1 ( \e z, s)
 h_1(z,s) \right],
\]
and by
 \[
  f_2( \bar m ( \pm \infty) ,\phi_1 (\xi) )  =  \frac 12  f''(\bar m ( \pm \infty))
\phi_1^2 (\xi),
\]
for $ \xi \in  \Om \setminus \NN (\e_0)$.

\vskip1.cm

\noindent
 { \bf n-th order term   in $ \e $ ($ 3\le n\le N $):}  
 \begin{equation}
\label{2.mm8}
  \begin{split}
\mu_{n-1} (  \e z, s)    =&  -
   h''_n(z,s) + a_n(z,s)   m'_0 (z)
     +
  \sum_{i=1}^{n-1} a_{n-i}(z,s)
 h_i'(z,s)
   \\ &  + \sum_{i=1}^{n-2} b_{n-i}(z,s) \frac {d^2}{d s^2} h_i (z,s)
  + \1_{\{n\ge 4\}} \sum_{i=1}^{n-3}   c_{n-i}(z,s)
\frac {d}{d s} h_i (z,s)
   \\& -    \e  \Delta
\phi_{n-1}( \e z, s)
 +f'(  m_0 ) \left [ h_n(z,s )  +   \phi_n ( \e z, s)   \right ] \\ &
  +
  f_n(m_0,m_1,m_2,..,m_{n-1})(\e z,s)   - G_{2,n-1} (\e z,s),  \qquad  \quad
\xi \in
  \NN (\e_0),
 \end{split}
  \end{equation}

 \begin{equation}
\label{2.ce8} \mu_{n-1} (\xi)  = -\e \Delta \phi_{n-1}( \xi )  +
f'(\pm 1) \phi_n( \xi) +f_n(\pm
1,\phi_1,\phi_2,..,\phi_{n-1})(\xi)    - G_{2,n-1} (\xi),
       \end{equation}
for $\xi \in \Om \setminus \NN (\e_0)$. \vskip0.5cm \noindent
\noindent {\bf Remainder term:} The remainder  $\tilde R_2(\xi,
t, \e)\equiv \tilde R_2$  is the following:
\begin{equation}
\label{v.10}
 \e \tilde R_2  =  \e^{N+1}G_{2,N} +  \e^{N+1}  A
+   \e^{N+2}\Delta \phi_N
   + E_1
   +  \e^{N+1}
B_{N+1}.
 \end{equation}
From \eqref{ap.100},  \eqref{st.2} and  \eqref{st.1} we obtain
\begin{equation}
\label{tue.1}
 \sup_{(\xi,t) \in \Om \times [0,T]} |\tilde R_2(\xi, t, \e) | \le C \e^N .   \end{equation}

\begin {rem}  Since, for $i=0, \cdots, N$,   we required $\partial_{\nu} G_{2,i} = 0$ on
$ \partial \Om$ then the $ \mu_i$ constructed in Theorem \ref
{re.2} satisfy $ \partial_{\nu} \mu_i =0$ on $ \partial \Om$; to
obtain $ \partial_{\nu} m_i =0$ on $ \partial \Om$ it is enough
to have $\partial_{\nu} \phi_i =0$ on $ \partial \Om$.
\end {rem}

\subsection   {Analysis of compatibility conditions}    \label {suba2}
Our aim is to analyze  the equations obtained so far in
 the previous subsection. The strategy is to find first at each order in $\e$
 the slowly varying  part $\phi_i$ by solving
the equations for $ \xi \in \Om \setminus \NN (\e_0,
\G^{(N)}_t)$. Then,  we extend $\phi_i$
   globally in $\Om$  and  determine  the rapidly
decaying part $h_i$ by solving the equations in $ \xi \in
  \NN (\frac {\e_0} 2, \G^{(N)}_t)$. However here, in order to continue to arbitrary order,
it is convenient to modify the way we extract the compatibility
condition required for solving the equation for $h_i$. The
modification is to add and subtract to each order a term of lower
order $ \e^{i+1} \alpha_i(s, \G) \bar m'(z)$, with
$\alpha_i(\cdot, \G)\in C^\infty  (\G)$.  Adding and subtracting
terms does not change, of course, the total quantity but it
modifies  the equation we obtain at each single order.
 In the following,
for the sake of short notation, we write  $\alpha_i (s) \equiv
\alpha_i (s,\G)$.

At  each order $i \ge 1$ in $\e$ the associate function  $m_i $
splits into two parts. The first is the function $\phi_i$ defined
globally in $\Om$ and satisfying Neumann condition on the
boundary of $\Om$ while the other part is the  function $h_i$
which differs from zero only in a tubular neighborhood of $\G $,
$\NN (\e_0, \G )$, and is exponentially decaying to zero far from
$\G$.

The zero order term is different in the sense that $m_0 $ far
from the interfaces relaxes exponentially fast to $\pm 1$. We
first state the following lemma, which is taken from \cite{A-B-C}.
We shall use this lemma to determine the condition for
solvability of equations of the type \eqref{4.4}, where $ \LL$ is
the operator on $L^2(\R)$ defined by
\begin{equation}
\label{elldef}  \LL g(z) = -g''(z) + f'(\bar m (z))g(z).
\end {equation}
  The operator $\LL$ is self adjoint on $L^2(\R)$ and has
a null space spanned by $\bar m'$. Therefore, the condition for
solvability of ${\mathcal L} h_1 = g$ is
\begin{equation}
\label{patt1} \int_\R g(z){\bar m }'(z){\rm d}z = 0.
\end{equation}

\begin{lem}\label {41} [see \cite{A-B-C}]  Let $A(z,s,t) $ defined for $z\in
\R$, $s \in  \G$, $t \in [0,T]$.   Assume that there exists
$A^{\pm} (s,t)$ such that $A(z,s,t)-A^{\pm} (s,t)= {\mathcal
O}(e^{-\alpha |z|}) $ as $|z| \to \infty$ for  $s \in  \G$ and $t
\in [0,T]$.   Then for each $s \in  \G$ and $t \in [0,T]$

\begin{equation}
\label{4.4}
 \begin{split}& (\LL w) (z,s,t) =
A(z,s,t) \qquad \hbox {{\rm for}} \qquad z \in \R ,\\ &
w(0,s,t)=0,\qquad w(\cdot,s,t)\in L^\infty (\R),  \end{split}
    \end{equation}
has a solution if and only if
\begin{equation}
\label{4.5}
 \int_{\R} A(z,s,t) \bar m'(z) dz =0  \qquad  \hbox {{\rm for all}} \qquad s \in \G, \qquad t \in [0,T].   \end{equation}

In addition if the solution exists, then it is unique and
satisfies
\begin{equation}
\label{4.6} D^{\ell}_z \left [ w(z,s,t)+\frac {A^{\pm} (s,t)}
{f'(1) } \right ] = {\mathcal O}(e^{-\alpha |z|} ) \qquad \hbox
{{\rm as}} \qquad |z| \to \infty, \qquad \hbox {{\rm for}}\qquad
\ell=0,1,2.
\end{equation}

Furthermore, if $A(z,s,t)$ satisfies
\[
 D^m_s D^n_t D^{\ell}_z \left [ A(z,s,t)- A^{\pm} (s,t)
\right] = {\mathcal O}(e^{-\alpha |z|} ),
\]
then
\[
 D^m_s D^n_t D^{\ell}_z \left [ w(z,s,t)+ \frac {A^{\pm}
(s,t)} {f'(1) } \right] = {\mathcal O}(e^{-\alpha |z|}),
\]
for all $m=0,1,\cdots,M$, $n=0,1,\cdots,N$, and
$\ell=0,1,\cdots,L+2$. Further, since $\LL$ is a preserving
parity operator,  the solution $w(z,s,t)$ is  odd (even)  with
respect to $z$ if $A(z,s,t) $ is  odd
 (even) with respect to $z$ for $s \in \G$ and $ t\in [0,T]$.  \end{lem}

\begin{rem}  Note that if $A(\cdot,\cdot,\cdot)\in C^\infty \left (
\R \times \G \times [0,T] \right ) $ then  the solution
 $w(\cdot,\cdot,\cdot)$ of the problem \eqref{4.4} is a function
 in $  C^\infty \left ( \R \times \G \times [0,T] \right )$.
This would be always the case whenever we  apply Lemma \ref{41}.
\end{rem}

\noindent  The compatibility conditions must hold for every $\G$
in $\MM$, and so, in our derivation we refer to $\G:=\G_t$.

\subsubsection{Zero order term in $ \e$:} For $\xi
\in  \NN (\frac {\e_0}2) $ using \eqref{2.221} we obtain
\begin{equation}
\label{2.aa1} 0 = -  \bar m''(z) + f( \bar m(z)) \qquad \hbox {\rm
for} \qquad  z \in  \Big{[}-\frac {\e_0} {2\e},\frac {\e_0} {2\e
}\Big{]},
\end{equation}
while \eqref{M.1} yields
 \begin{equation}
\label{2.b1}
 0 =   f ( \pm 1)
\quad \hbox { for }\quad \xi \in \Om \setminus \NN (\e_0).
\end{equation}   The above relations determine $m_0$, more specifically $m_0(z) =
\bar m(z)$,  where $\bar m$ solves the equation
$$ -   m''(z) + f(m(z)) =0, \;\;\;\;z \in \R.$$

\subsubsection{First  order term   in $ \e $ and determination of $V_0$ via the Fredholm alternative}
As explained at the beginning of this section, it is convenient
for solving  \eqref{2.20} to add a  term $ \e \alpha_1 (s,\G)
\bar m' (z)$, $s \in \G$ and $ z \in \R$,  with $\alpha_1
(\cdot,\G) $ to be determined. This term  will be subtracted to
the second  order.  Recalling the definition of $\LL$, see
\eqref{elldef}, and adding $\e \alpha_1(s) \bar m' (z)$, we  write
\eqref{2.20} as
\begin{equation}
\label{2.a2}
  \mu_{0}(\e z,s)- f'(\bar m(z))
    \phi_1( \e z, s)-  K(s)\bar m'(z)  +\e \alpha_1 (s) \bar m' (z) +G_{2,0} ( \e z, s) =
 (\LL h_1)(z,s),      \end{equation}
  for $  s
\in \G$,  $|z|\le \frac {\e_0} {2 \e}$. By \eqref{2.200} we obtain
\begin{equation}
\label{4.1}
 \phi_1(\xi)=   \frac {  \mu_{0}( \xi) + G_{2,0} ( \xi)} { f'( \bar m ( \pm \infty ))}, \qquad   \xi \in
\Om \setminus \NN (\e_0).
 \end{equation}

 We extend this definition of
$\phi_1$  globally in $\Omega$. We then insert   \eqref{4.1} into
\eqref{2.a2} obtaining for $  s \in \G$,  $|z|\le \frac {\e_0} {2
\e}$
\begin{equation}
\label{2.aa2}
   \mu_{0}(\e z,s) \left [1   -   \frac  {f'(\bar m(z))}
  {f'(\pm1)}  \right ] +
\left(1 - \frac {f'( \bar m(z))} { f'(\pm1)}\right) G_{2,0}(\e
z,s)     - K(s) \bar m'(z) +\e \alpha_1 (s) \bar m' (z)    = (\LL
h_1)(z, s).
      \end{equation}
Since the left hand side of \eqref{2.aa2} tends exponentially to
$0$ as $z \to \pm \infty$, if the solution of \eqref{2.aa2}
exists (cf.  Lemma \ref{41}), then it decays exponentially  fast
to $0 $. We can, therefore, extend \eqref{2.aa2} for any $ z$ in
$\R$.

 We prove the next result.

\begin{lem}\label {42}  Set

\begin{equation}
\label{z.1}
     V^{(0)}_0(\xi,\G) =    {\mathcal T}_\G \left [
    S \left \{ K(\cdot) - \frac {\int_\G   K (\xi) \dsx}  {|\G|} \right \}
  + \frac 1 4   \left \{ B_{0,0,0}(\cdot) - \frac {\int_\G
B_{0,0,0}(\xi) \dsx }   {|\G|}  \right \} \right ] ( \xi ) \qquad
\xi \in \G ,
\end{equation}
 where
 \begin{equation}
\label{surf} S= \frac 14 \int_\R \left (\bar m'(z) \right)^2
dz,   \end{equation} and $B_{0,0,0}(\cdot)$ is defined in \eqref
{roma.3} while ${\mathcal T}_\G $ is the Dirichlet-Neuman
operator given by \eqref{roma4}. Then there exists a uniquely
determined $\alpha_1(\cdot,\G) \in C^\infty (\G)$ and a unique
solution $h_1(\cdot,s)$ of \eqref{2.aa2} with $s \in \G$,
   such that
 $h_1(0,s)=0$ and $h_1(\cdot,s) \in L^\infty (\R)$.
  Moreover,  $h_1(\cdot , s)  $ for $s \in \G$  and its
derivatives  with respect to $z$ decay exponentially fast to $0$
as $z$ tends to $\pm \infty$.  Further $h_1(\cdot , s)=    \tilde h_1(\cdot , s) + \e q_\e(\cdot,s)$
where $\tilde h_1(\cdot , s)$ is an even function of $z$,  $\tilde h_1(0,s)=0$ and $ \tilde h_1(\cdot,s) \in L^\infty (\R)$.
\end{lem}
\begin{proof}
We start  determining $\tilde h_1$ as solution of 
\begin{equation}
\label{2.d2}
   \mu_{0}(0,s) \left [1   -   \frac  {f'(\bar m(z))}
  {f'(\pm1)}  \right ] +
\left(1 - \frac {f'( \bar m(z))} { f'(\pm1)}\right) G_{2,0}(0,s)     - K(s) \bar m'(z) +\e \tilde \alpha_1 (s) \bar m' (z)    = (\LL
\tilde h_1)(z, s).
      \end{equation}
     Equation  \eqref {2.d2} differs from \eqref {2.aa2}  only for terms of order $\e$.
     Namely  $|  G_{2,0}(0,s) -  G_{2,0}(\e z,s)| \le
C \e |z|$   and $ |\mu_{0}(0,s)-  \mu_{0}(\e z ,s)| \le
C \e |z|$.
 For any fixed $ s  \in \G $, the condition for the
existence of $\tilde h_1$ requires that
\begin{equation}
\label{4.7d}
  \begin{split} & \int_\R \mu_{0}(0,s) \left [1   -   \frac  {f'(\bar m(z))}
  {f'(\pm 1)}  \right ]   \bar m'(z) dz      +  \int_\R G_{2,0}(0,s) \left [1   -   \frac  {f'(\bar m(z))}
  {f'(\pm 1)}  \right ]   \bar m'(z) dz
\\ &  =  \left [ K(s) -\e\tilde  \alpha_1 (s) \right ]  \int_\R
\left (\bar m'(z) \right)^2 dz\;\;\;\;\hbox {for}\;\;\;\; s  \in
\G.
\end{split}  \end{equation}
Taking into account that    $
 \int_\R f'(\bar m(z))\bar m'(z) dz = f(1)- f(-1)= 0  $  and $ \int \bar m'(z) dz =2$ formula  \eqref {4.7d} can be written as 
 \begin{equation}
  2 [ \mu_{0}(0,s)       +   G_{2,0}(0,s)]  =
    \left [ K(s) -\e\tilde  \alpha_1 (s) \right ]  \int_\R
\left (\bar m'(z) \right)^2 dz\;\;\;\;\hbox {for}\;\;\;\; s  \in
\G.
  \end{equation}

  Recalling \eqref {peg1as1} we obtain

\begin{equation}
\label{4.7a}
  \begin{split} & 2  \mu_{0,0}(0,s)     = -   2 \tilde \mu_{0}(0,s)  -  2 G_{2,0}(0,s) 
 \cr &  
+   \left [ K(s) -\e \tilde \alpha_1 (s) \right ]  \int_\R \left (\bar
m'(z) \right)^2 dz\;\;\;\;\hbox{for}\;\;\;\; s  \in \G.
\end{split}  \end{equation}

 Further, we set
\begin{equation}
\label{roma.1} \mu_{0,0,0}( \xi) = 2\int_{\G } V^{(0)}_0( \eta )
G( \xi,
 \eta )  \dse    + c_0(t)  \qquad \xi \in \Om , \end{equation}

 \begin{equation}\label{roma.22} \tilde \mu_{0,0,0}(\xi) = 2  \langle V_0\rangle_{\eightpoint
\Gamma} \int_{\G }  G( \xi,
 \eta )  \dse   -  \int_\Om G(\xi,\eta)G_{1,0} (\eta)  {\rm d}
 \eta
  \qquad \xi \in \Om, \end{equation}
  and
 \begin{equation} \label{roma.3}  B_{0,0,0} (s)     =
-2 \left [ \tilde\mu_{0,0,0} (0,s) +  G_{2,0}(0,s) \right ].
   \end{equation}
It is immediate to see that $$ \mu_{0,0}(  \xi) -\mu_{0,0,0}( \xi)
\simeq  \e,$$  $$\tilde \mu_{0}(\xi)-  \tilde \mu_{0,0,0}(\xi)
  \simeq \e .$$          We   first choose $V^{(0)}_0$
 imposing   for $
s  \in \G$, cfr \eqref {4.7a}, the next identity
  \begin{equation}  
\label{4.8}
 2 \mu_{0,0,0}( \xi)     =   B_{0,0,0} (\xi)+  K(\xi )   \int_\R
\left (\bar m'(z) \right)^2 dz \quad \xi \in \G.
   \end{equation}
 Inserting \eqref{roma.1}   in \eqref{4.8}
  and   integrating over $\G$  we obtain  that
 \begin{equation}
\label{4.8E}
 4\int_{\G} \dsx \int_{\G } V^{(0)}_0(\eta) G( \xi,
 \eta)  \dse  +  2c_0 (t) |\G|    =   \int_\G   K(\xi) \dsx   \int_\R
\left (\bar m'(z) \right)^2 dz  +      \int_{\G}  B_{0,0,0}
(\xi)\dsx ,
     \end{equation}
 and therefore,
\begin{equation}
\label{4.8G}
  c_0 (t)=  \frac 1  {2|\G|} \left [    \int_\G   K(\xi) \dsx  \int_\R
\left (\bar m'(z) \right)^2 dz       +  \int_{\G}
B_{0,0,0}(\xi)\dsx - 4\int_{\G }
 \dsx  \int_{\G } V^{(0)}_0( \eta ) G(\xi,
 \eta) \dse \right ] .  \end{equation}
 We  note that  $c_0(t)$   is written in terms  of
the velocity field  $V^{(0)}_0$ which still needs to be
determined.   We insert $c_0(t)$,  as in \eqref{4.8G}, into
\eqref{4.8} to obtain the   equation determining $V^{(0)}_0$
$$
{\mathcal S}_{\G} V^{(0)}_0 (\eta,\G) = S \left [
    K(\eta)  - \frac { \int_\G K (\xi) \dsx }   {|\G|} \right]
    +    \frac 14  \left [  B_{0,0,0}(\xi)-   \frac { \int_{\G}  B_{0,0,0} (\xi)\dsx} {|\G|} \right ]\;\;\;\;\eta \in
    \G.$$
   Here, ${\mathcal S}_{\G}$ is   the linear operator  defined in
\eqref{sform} and $S$ the quantity defined in \eqref {surf}.
Applying the Dirichlet-Neumann operator, see \eqref{ma.3},  we
arrive at \eqref {z.1}.
 The above determines first $V^{(0)}_0$ and then $c_0(t) $, see \eqref{4.8G}.
  Now since $V^{(0)}_0$   and  $c_0(t) $ are  chosen, we may then
simply choose $ \tilde \alpha_1(s)$  so that \eqref{4.7d} is satisfied.
Then   for any $ s \in \G$,  the existence of a unique $\tilde h_1$
satisfying $\tilde h_1(0,s)=0$ is assured; $\tilde h_1$ is exponentially
decaying to zero as $|z| \to \infty$. Since the left-hand side of
\eqref{2.d2} is even, then the solution $\tilde h_1 (\cdot,s) $ is even
as a function of $z$.

Define $   \e q_\e(\cdot,s)= h_1 (\cdot,s) -\tilde h_1 (\cdot,s)$ and subtract  \eqref {2.d2} to \eqref {2.aa2}. 
We have that   $ q_\e(\cdot,s)$    satisfies 
  \begin{equation} 
\label{2.c2} \begin {split} & \frac 1 \e 
 \left [1   -   \frac  {f'(\bar m(z))}
  {f'(\pm1)}  \right ]  [  \mu_{0}(\e z,s)-  \mu_{0}(0,s)   +
   G_{2,0}(\e z,s)- G_{2,0}(0,s)]      + [ \alpha_1 (s) - \tilde \alpha_1 (s)] \bar m' (z)  \cr &   =   (\LL
  q_\e)(z, s), \end {split}
      \end{equation}
       where $ \alpha_1 (\cdot)$ must still be determined. 
    We determine   $ \alpha_1 (\cdot)$  so that  the  following solvability condition for  $  q_\e(\cdot,s)$ holds:
  \begin{equation} 
  \begin {split} &  \frac 1 \e \int    dz \bar m' (z)   \left [1   -   \frac  {f'(\bar m(z))}
  {f'(\pm1)}  \right ]  [  \mu_{0}(\e z,s)-  \mu_{0}(0,s)   +
   G_{2,0}(\e z,s)- G_{2,0}(0,s)]   \cr &   +  [ \alpha_1 (s) - \tilde \alpha_1 (s)] \int dz  (\bar m' (z))^2 =0. 
  \end {split}    \end{equation} 
By Lemma \ref {41} we    have  that  for all $s \in \G$,    $  q_\e(0,s)=0$, $  q_\e(0,s) \in L^\infty(\R)$.
\end{proof}
\begin {rem} \label {nov1}  Let us denote by
$$ \mu_{0,0,0}^F =  \mu_{0,0,0}+ \tilde  \mu_{0,0,0}, $$
the quantities defined in \eqref {roma.1} and \eqref {roma.22}.
Taking into account  \eqref {2.1000} and  the definition of
$\langle V_0\rangle_{\eightpoint \Gamma} = \frac 1 {2 |\G_t|}
\int_\Om  G_{1,0} (\eta)  {\rm d}\eta $,   see \eqref {zcond1},
we  have that  $\mu_{0,0,0}^F$  is the solution of 
\begin {equation} \label{mv1} \Delta \mu (\xi) = - G_{1,0} (\xi) \qquad{\rm for}\qquad \xi
\in \Omega \setminus \G,
\end{equation}
\begin{equation}
\label{mv2} \mu (\xi) =  2 S K(\xi)  -G_{2,0} (\xi) \
\;\;\;\; {\rm on}\;\;\;\;\G,\;\;\;\;
 \partial_ \nu \mu = 0\;\;\;\;{\rm
 on}\;\;\;\;
\partial\Omega.
\end{equation}
One deduces \eqref {mv2} by   \eqref {4.8}, taking into account the definition of $S$
and      \eqref {roma.3}.

\end {rem}
\vskip0.5cm
 \subsubsection{Second  order term   in $ \e $ }
\vskip0.5cm \noindent
  From \eqref {2.c8} we have that \begin{equation}
\label{4.10} \phi_2(\xi)= \frac 1 {f'(\pm 1) } \left[
 \mu_1( \xi ) -
f_2(\bar m ( \pm \infty),\phi_1  (\xi)) +\e \Delta \phi_1 +
G_{2,1} (\xi)   \right ], \qquad   \xi \in \Om \setminus \NN
(\e_0).   \end{equation}
 As was done before, we  extend  the validity of \eqref{4.10}  globally in $\Om$.
We insert  \eqref{4.10} into \eqref{2.m8}. Further, we add and
subtract to the next order $\e \alpha_2(s) \bar m'(z)$ to obtain
\begin{equation}
\label{4.11} \mu_1 (\e z,s) \left [1   - \frac{ f'(\bar m (z))}
{f'(\pm 1) } \right] - A_2(z,s) + \e \alpha_2(s) \bar m'(z)
 =
 (\LL h_2)(z, s),
   \end{equation}
 where we set
 \begin{equation}
\label{ap.5}
  \begin{split}  A_2(z,s)  = & \left [ G_{2,1} (\e z,s)+ \e \Delta  \phi_1(\e z,s) - f_2(\pm 1,\phi_1) (\e z,s) \right ]  \left [1-   \frac {f'(\bar m (z))}{
f'(\pm 1)}  \right ]
 \\ &
- K^2(s) z \bar m'(z)-K (s)  h'_1(z,s) +\alpha_1(s) \bar m'(z) .
\end{split}       \end{equation}
 All the quantities in  \eqref{ap.5} have been already determined.
Furthermore,
\begin{equation}
\label{exp}
 \lim_{|z| \to \infty} A_2(z,s) = 0 \qquad  s \in \G,   \end{equation}
the convergence being exponentially fast.
   As done
before, we extend \eqref{4.11} in $\R$.
\begin{lem}\label {43}
Set
\begin{equation}
\label{z.1A}
     V^{(0)}_1(\xi,\G) =   {\mathcal T}_{\G} \left [  \frac 14 B_1( \cdot )   -  \frac 1{4|\G|} \int_{\G}
B_1( s )  {\rm {d} } s  \right ] (\xi)
  \qquad \xi \in \G,   \end{equation}
where $ B_1(s ) $ is defined in \eqref{o.5} and ${\mathcal T}_\G$
is the Dirichlet--Neumann operator. Then there exist uniquely
determined $\alpha_2(\cdot,\G) \in C^{\infty} (\G)$ and
$h_2(\cdot,s) \in \L^{\infty} (\R)$ with $h_2(0,s)=0$ for $ s \in
\G $, solutions  of \eqref{4.11}.
 Moreover $h_2 (\cdot,s)$  and its
derivatives with respect to $z$
  decay   exponentially  to  $0 $,  as $z$ tends to $\pm
\infty$.   \end{lem}
\begin{proof}
The solvability condition, see \eqref{4.5}, is satisfied
provided  that for $s  \in  \G$ and $t \in [0,T]$ the next
relation holds true
\begin{equation}
\label{4.12}
 \begin{split} \int_\R  \mu_1 (\e z,s ) \left [1   -   \frac  {f'(\bar m(z))}
  {f'(1)}  \right ]   \bar m'(z) dz  =  \int_\R A_2(z,s)
\bar m'(z) dz  - \e \alpha_2(s) \int_\R \left (\bar m'(z)
\right)^2 dz,    \end{split}
\end{equation}
  where  $\mu_1 $ is defined in \eqref{div}. The term $\tilde \mu_{1} $ of
$\mu_{1} $ has been  already completely determined. As in the
previous case, still to be determined are the constant $c_1(t)$,
the velocity $V^{(0)}_1$ and $\alpha_2(s)$. First, we write
\eqref{4.12} as
\begin{equation}
\label{o.4}
 \int_\R\mu_{1,0}(\e z,s ) \left [1   -   \frac  {f'(\bar m(z))}
  {f'(1)}  \right ]   \bar m'(z) dz =  B_1(s) -\e \alpha_2(s) \int_\R
\left (\bar m'(z) \right)^2 dz,
\end{equation}
where
\begin{equation}
\label{o.5}
   B_1( s)   =
  \int_\R \left \{   A_2(z,s) -   \tilde \mu_{1}(\e z,s  )  \left [1   -   \frac  {f'(\bar m(z))}
  {f'(1)}  \right ] \right \}   \bar m'(z) dz.    \end{equation}
Let now
\begin{equation}
\label{roma.2} \mu_{1,0,0}( \xi): = 2\int_{\G } V^{(0)}_1( \eta  )
G( \xi, \eta )  \dse +c_1(t) \qquad   \xi \in \Omega.
\end{equation}
Since $\mu_{1,0,0}(\xi) - \mu_{1,0}( \xi) \simeq \e $, we choose
$V_1$ by imposing
\begin{equation}
\label{o.4b}
  \int_\R\mu_{1,0,0}(0, s) \left [1   -   \frac  {f'(\bar m(z))}
  {f'(1)}  \right ]   \bar m'(z) dz =  B_1( s),
\end{equation}
and obtain
\begin{equation}
\label{o.4c}
   \mu_{1,0,0}(0,s)= \frac 12 B_1( s)  \qquad s \in \G.
\end{equation}
 Inserting \eqref{roma.2} in \eqref{o.4c} and integrating over $\G$ we arrive at
 \begin{equation}
\label{(o.4d)} c_1(t) = \frac 1 {|\G_t|} \left [ \frac 12
\int_{\G_t} B_1(\eta) \dse  - 2 \int_{\G_t } \dsx \int_{\G_t }
V^{(0)}_1( \eta  ) G( \xi,
 \eta)  \dse \right ].
\end{equation}
Observe that $ \int_{\G} V^{(0)}_1(s) ds =0$ and let $S_{\G}$ be
the linear operator  defined in
 \eqref{sform}.
Then obviously, \eqref{o.4c} can be written as
\[
S_{\G} V^{(0)}_1 (\xi)=   \frac 14 B_1(\xi )   -  \frac 1{4 |\G|}
\int_{\G} B_1(s) ds
 \qquad \xi \in \G,
 \]
 and thus, applying the Dirichlet-Neumann operator (see \eqref{ma.3}) we
obtain  \eqref{z.1A}. This determines  the (constant in $\xi$)
$c_1(t) $. Now, since $V^{(0)}_1$ and $c_1(t) $ are determined we
may choose $ \alpha_2(s)$ so that \eqref{o.4} is satisfied.
   \end{proof}

\begin{rem}
Notice that $\mu_{1,0,0}$ solves
\begin{equation}
\label{4.16} \begin{split} & \Delta  \mu_{1,0,0} =0 \qquad \hbox
{{\rm for}}\qquad \xi \in \Omega \setminus \G,  \\ & \mu_{1,0,0}
(\xi) =  \frac 1 2 B_1(\xi)\qquad
   \hbox {{\rm on}}\qquad  \G.    \end{split}   \end{equation}
\end{rem}

\subsubsection {  $n-$th order term   in $ \e $, $ 3 \le n\le N$:} As previously,
  we determine  the function $\phi_n$ for $ \xi
\in \Omega \setminus \NN (\e_0)$
 from
\eqref{2.ce8}. Then, we extend the validity in $\Omega$ obtaining
\begin{equation}
\label{2.c08}
  \phi_n(
\xi)=  \frac 1 { f'(1) } \left [ \mu_{n-1} ( \xi )    +\e\Delta
\phi_{n-1}( \xi )  -f_n(\pm
1,\phi_1,\phi_2,\cdots,\phi_{n-1})(\xi) + G_{2,n-1} (\xi) \right]
\qquad \xi \in \Omega.
    \end{equation}
We then insert \eqref{2.c08} into \eqref{2.mm8}, we add and
subtract to  the next order  the quantity $\e \alpha_n(s) \bar
m'(z)$, to the left hand side of \eqref{2.mm8} and we obtain
\begin{equation}
\label{4.final} \mu_{n-1} (\e z,s ) \left [ 1- \frac {f'(\bar
m(z))} {f'( 1)} \right]  - A_n(z, s) + \e\alpha_n(s) \bar m'(z)  =
      (\LL h_n)(z,s),
\end{equation}
where  we set
 \begin{equation}
\label{2.ZZ8}
 \begin{split} &  A_n(z, s)=
      -a_{n-1}(z,s) \bar m'-
  \sum_{i=1}^{n-1} \left[ a_{n-i}(z,s) h_i'(z,s)
\right] - \sum_{i=1}^{n-2} b_{n-i}(z,s) \frac {d^2}{ds^2} h_i
(z,s)
    \\ &   - \1_{n\ge 4}\sum_{i=1}^{n-3} \left[ c_{n-i}(z,s)  \frac {d}{ds} h_i
(z,s) \right]   -\e  \Delta \phi_{n-1}(\e z,s)\left[1- \frac
{f'(\bar m (z))} {f'( 1)} \right]
        +    G_{2,n-1} (\e \xi)\left[1- \frac {f'(\bar m (z))} {f'(
1)} \right]   \cr & + \frac {f'(\bar m(z))} {f'( \pm1)} f_n(\pm
1,\phi_1,\phi_2,\cdots,\phi_{n-1})(\e z,s)   -
  f_n(m_0,m_1,m_2,\cdots,m_{n-1})(\e z,s).
     \end{split}  \end{equation}
  It is easy to verify that  for all $s \in \G$
 \[
\lim_{|z| \to \infty}  A_n(z,s) =0,
\]
where the convergence is exponentially fast.  Namely there is no
problem for those terms involving $\bar m'$, $h_i (\cdot,s)$ and
their derivatives because of the exponential convergence to zero
of all these terms for all $s \in \G$. For the remaining terms
recall that $\displaystyle{\lim_{|z| \to \infty}} f'(\bar m (z)) =
f'(\pm 1) $,  $m_i= h_i + \phi_i$ with $h_i(z, s)  \to 0$ as $|z|
\to \infty$  for all $s \in \G$, all limits being exponentially
fast. Then one obtains immediately
\[
\lim_{|z| \to \infty} \left [  1- \frac {f'(\bar m(z))} {f'(
\pm1)} \right ]   f_n(\pm1,\phi_1,\phi_2,\cdots,\phi_{n-1})(\e
z,s) = 0,
\]
exponentially fast also. We  extend  \eqref{4.final} to hold on
all of $\R$ and regard it as an equation for $ h_n (\cdot,s) $
for $s \in \G$.

\begin{lem}\label {44}
For any positive  integer $n$, $n \le N$, set
\begin{equation}
\label{zz.1} V^{(0)}_{n-1}(\xi , \G ) =   {\mathcal T}_\G \frac
14 \left [B_{n-1}(\cdot)- \frac 1 {|\G|} \int_{\G}B_{n-1}(s)  ds
\right ] \qquad  \hbox {for} \qquad
 \xi \in
\G,
  \end{equation}
where $ B_{n-1}(s) $ is defined in \eqref{EP.1}.  Then there
exist   uniquely determined  $ \alpha_n(\cdot,\G) \in C^\infty
(\G)$  and
 $h_n(\cdot,s)\in L^{\infty} ( \R)$ for $s \in \G$ with
 $h_n(0,s)=0$,
solutions of \eqref{4.final}. Moreover, $h_n (\cdot,s)  $ for all
$s \in \G$, and its derivatives with respect to $z$ decay
exponentially
  to $0 $   as $z \to \pm \infty$.\end{lem}
\begin{proof}
The solvability condition is satisfied provided that
\begin{equation}
\label{L.20} \begin{split}
 \int_\R
\mu_{n-1}(\e z,s) \left [ 1   -   \frac  {f'(\bar m(z))}
  {f'(1)}  \right ]   \bar m'(z) dz  =   \int_{\R} A_n(z, s)
\bar m'(z) dz   -\e \alpha_n(s)   \int_\R \left (\bar m'(z)
\right)^2 dz.    \end{split}   \end{equation} Since in view of
\eqref{rep10}
  $
\mu_{n-1}= \mu_{n-1,0}+ \tilde \mu_{n-1} $ with $\tilde
\mu_{n-1}$ being  already   determined  to satisfy \eqref{L.20},
we require  that
\begin{equation}
\label{EP.2}
 \int_\R \mu_{n-1,0}(\e z, s) \left [1   -   \frac  {f'(\bar m(z))}
  {f'(1)}  \right ]   \bar m'(z) dz
  = B_{n-1} (s)  -\e \alpha_n(s)  \int_\R
\left (\bar m'(z) \right)^2 dz,
   \end{equation}
where
\begin{equation}
\label{EP.1}
  B_{n-1} (s)  =   \int_{\R}\left \{  A_n(z,s) - \tilde \mu_{n-1} (\e z,s) \left [ 1   -   \frac  {f'(\bar
m(z))}
  {f'(1)}  \right ]   \right\}
\bar m'(z) dz.
\end{equation}
We set
\begin{equation}
\label{ap.51}
  \mu_{n-1,0,0}(\xi) =
2\int_{\G } V_{n-1}( \eta ) G(\xi,
 \eta) \dse+ c_{n-1} (t).   \end{equation}
Since $ \mu_{n-1,0}(\xi, t)- \mu_{n-1,0,0}(\xi, t) \simeq \e $
then we may determine $ V_{n-1}$ by imposing
\begin{equation}
\label{EP.22}
 \int_\R \mu_{n-1,0,0}(0,s) \left [1   -   \frac  {f'(\bar m(z))}
  {f'(1)}  \right ]   \bar m'(z) dz
  = B_{n-1} (s),
  \end{equation}
obtaining thus
\begin{equation}
\label{EP.2a}
 \mu_{n-1,0,0}(0,s)  = \frac 12 B_{n-1} (s)  \qquad s\in \G.    \end{equation}
Inserting  \eqref{ap.51} in \eqref{EP.2a} and integrating over
$\G$ we get
\begin{equation}
\label{constant}
 c_{n-1} (t) = \frac 1 {|\G|} \left [\frac 12 \int_{\G} B_{n-1} (\eta ) \dse  - 2  \int_{\G}
\dsx \int_{\G } V^{(0)}_{n-1}( \eta) G(\xi,
 \eta)  \dse  \right ].  \end{equation}
We insert \eqref{constant} into \eqref{ap.51}, so, \eqref{EP.2a}
gives
 \[
  S_{\G} V^{(0)}_{n-1} (\xi) = \frac 14 \left [  B_{n-1} (\xi  )- \frac 1 {|\G|}  \int_{\G}
B_{n-1} (s) {\rm d} s \right ] \qquad \xi \in \G,
 \]
and \eqref{zz.1} follows.  This determines $V^{(0)}_{n-1}$ first
and then $ c_{n-1}(t)$. Therefore, we may define $\alpha_n$ in
order to satisfy \eqref{EP.2}. \end{proof}

\subsection { Proof of Theorem  \ref{re.3}}

To complete the proof of  Theorem \ref{re.3}  we need to estimate
the remainder term (cf. \eqref{v.10}) given by
\begin{equation}
\label{v.100}
 \begin{split} & \e R_2(\xi, t, \e) =  \e \tilde R_2(\xi, t, \e) - \e^{N+1} \alpha_N(s(\xi),t) \bar m'( \frac
{d(\xi,\G)} {\e}).    \end{split}
 \end{equation}
Since \eqref{tue.1} holds true, we obtain that
\begin{equation}
\label{Z.4} \sup_{\xi \in \Omega} \sup_{t \in [0,T]} | R_2(\xi,t,
\e) | \le C \e^{N}.   \end{equation} So, Theorem \ref{re.3} is
proved. \qed

\section{Proof of Theorem \ref {1.1}}
The  proof of   Theorem \ref {1.1}  is an immediate consequence
of the following  result.

\noindent\begin{thm} \label{1}Take $N>1$ and  $G_1$ and $G_2$   as in Assumption A1. There exist vector fields $V_j$,
$j=0,\cdots,(N-1)$ and  functions $m_j$, $j=0,\cdots, N$  from ${\mathcal M}$ to  $C^\infty(\Om)$  as in
the Ansatz \ref{main} such that the following holds.  For any
$\G_0 \in \MM$, choose $k_0 \ge k (\G_0)$, set $\e_0 = \frac 1 {2
k_0}$  and let $T$  be the  lifetime of the solution of
\eqref{e0} in ${\mathcal  M}$, according to  \eqref{ll.2}. Then
for all $t < T$ and
   for all
$\e \in  (0, \e_0]$  we can construct   $ (\tilde m^{(N)}, \tilde
\mu^{(N-1)}) \in C^\infty (\Om \times [0,T]) $  where $\tilde
m^{(N)}$ is an $\e^N$ modification of $m^{(N)}$, i.e.
$\sup_{(\xi,t)\in \Om \times [0,T]} | \tilde m^{(N)}(\xi,t)
-m^{(N)}(\xi,t)| \le C \e^N $ and $\tilde \mu^{(N-1)}$ is an
$\e^{N-1} $ modification of $\mu^{(N-1)}$, i.e.  $
\sup_{(\xi,t)\in \Om \times [0,T]} |\tilde \mu^{(N-1)}(\xi,t)-
 \mu^{(N-1)}(\xi,t)| \le C\e^{N-1} $
satisfying

\begin{equation}
\label{v.5}  \begin{split}   &  \partial_t 
\tilde m^{(N)}(\xi,t)   =
   \Delta
\tilde \mu^{(N-1)}(\xi, t) + \sum_{j=0}^{N-1}  \e^j G_{1,j}(\xi)
\quad  \hbox { in }\quad \Omega \times (0,T),
\\ &
 \tilde \mu^{(N-1)}(\xi, t)  = -\e \Delta   \tilde m^{(N)}(\xi,t)
   +\frac 1 \e  f(\tilde m^{(N)}(\xi,t))    - \sum_{j=0}^{N-1}  \e^j G_{2,j}(\xi) +  R^{(N)}(\xi,t, \e)
\quad   \hbox { in }\quad
 \Omega \times (0,T),     \end{split}  \end{equation}
 where
\[  \sup_{\xi \in \Om} \sup_{t \in [0,T]} \left |  R^{(N)}(  \xi,
t, \e) \right | \le C \e^{N-1} .  \]

 Further, $\tilde \mu^{(N-1)}(\cdot,t)$ and $\tilde m^{(N)}(\cdot,t) $, for $ t
\in [0,T] $,  satisfy Neumann homogeneous  boundary conditions on
the boundary of $\Om$.   In addition
\begin{equation}
\label{v.6} \sup_{t \in [0,T]} \sup_{\xi \in\Om} | \tilde
\mu^{(N-1)}(\xi, t) - \mu_{0,0,0}^F(\xi,t) | \le C\e,
   \end{equation}
where $\mu_{0,0,0}^F$ is the solution of
\begin{equation}
\label{no.2} \Delta\mu (\xi) = -G_{1,0} (\xi)   \qquad \xi \in
\Om\setminus \G^{(N)}_t,
\end{equation}
subject to the boundary conditions
\begin{equation}
\label{ms2b} \mu (\xi) =  2S K(\xi) -   G_{2,0} (\xi)  \quad
\hbox {\rm on}\quad \G^{(N)}_t,\;\;\;\;
  \partial_ \nu \mu  = 0  \quad \hbox {\rm on}\quad
\partial \Om,
\end{equation}
and
\begin{equation}
\label{v.7}\sup_{ t \in [0,T]} \sup_{\xi \in \NN (\e_0,
\G^{(N)}_t)} \left |   \tilde m^{(N)}(\xi,t) - \bar m\left (
\frac {d (\xi,\G^{(N)}_t) }\e    \right  )\right | \le C \e  ,
\end{equation}

\begin{equation}
\label{v.8} \sup_{t \in [0,T]}\sup_{\xi \in \Om \setminus \NN
(\e_0), \G^{(N)}_t)} \left |   \tilde m^{(N)}(\xi,t) \mp  1
\right | \le C \e \ .
\end{equation}
 \end{thm}
\begin {proof}

 Set
\begin{equation}
\label{v.3}
 \tilde m^{(N)}(\xi, t)= m^{(N)}(\xi, t)-
\int_0^t \bar R_1 (\t ,\e)  d \t,   \end{equation} where
\[ \bar R_1 (t,\e)=  \frac 1 {|\Om|}
\int_\Om R_1(\xi, t, \e) {\rm d }\xi,
\]
and  $  R_1(\xi, t, \e)$  is the remainder in Theorem \ref{re.2},
defined in \eqref{remainder} and estimated in \eqref {2.014} and
\eqref {2.015}. Let us denote by
\begin{equation}
\label{v.4}
 \tilde \mu^{(N-1)}(\xi, t)= \mu^{(N-1)}(\xi, t)+ v(\xi, t),
 \end{equation}
where $ v(\xi, t)$ solves
\begin{equation}
\label{z.10}
 \begin{split} & \Delta v(\xi, t)=  R_1(\xi, t, \e)- \bar R_1(t,\e)   \qquad  \hbox
{for} \qquad \xi \in \Omega,  \\ &  \partial
_\nu  v= 0  \qquad \hbox {on} \qquad \partial \Om , \end{split}
\end{equation}  with the further requirement
\[
 \int_\Om v (\xi, t) {\rm d } \xi=0 \qquad t \in [0,T].
\]
 Since $ |R_1(\xi, t, \e)|\le C(T)\e^{N-1}$ we have that  $ |v(\xi,
t)| \le C\e^{N-1} $. The functions  $ \tilde m^{(N)} $ and $\tilde
\mu^{(N-1)}$ satisfy \eqref{v.5}. Namely  the first equation of
\eqref{v.5} is satisfied by Theorem \ref{re.2} and by
construction, see \eqref{v.3} and  \eqref{z.10}.  The second
equation is obtained   from Theorem \ref{re.3} adding and
subtracting terms to obtain $\tilde \mu^{(N-1)}$ and $\tilde
m^{(N)}$. We obtain
\[
 \tilde \mu^{(N-1)}   = \mu^{(N-1)} + v =
-\e \Delta  \tilde m^{(N)}  + \frac 1 \e f(\tilde m^{(N)} ) +
R^{(N)},
\]
where
\[
   R^{(N)} \equiv  R^{(N)} (\xi,t,\e)= \frac 1 \e \left [  f \left (\tilde m^{(N)}+ \int_0^t
\bar R_1(\t,\e) {\rm d }\t \right ) - f(\tilde m^{(N)}) \right ]
+ R_2 + v,
 \]
and $R_2$  is the remainder in Theorem \ref{re.3}, see
\eqref{v.100}. Since $ \bar R_1 ={\mathcal O}(\e^N)$, $R_2=
{\mathcal O}(\e^{N})$,  $v= {\mathcal O}(\e^{N-1})$ and
\begin{equation}
\label{NY.2}
 \frac 1 \e \left [  f \left (\tilde m^{(N)}+ \int_0^t \bar
R_1(\t,\e) {\rm d }\t \right ) - f(\tilde m^{(N)}) \right ] \le
\frac C \e \int_0^t \bar R_1(\t,\e) {\rm d }\t = {\mathcal
O}(\e^{N-1}),  \end{equation} then the second equation of
\eqref{v.5} is satisfied as well.  In Remark \ref {nov1}   it is
explained that  if  $\mu_{0,0,0}^F = \mu_{0,0,0}+ \tilde
\mu_{0,0,0} $, where $ \mu_{0,0,0}$ and  $\tilde \mu_{0,0,0}$ are
the quantities defined in \eqref {roma.1} and \eqref {roma.22}
then it verifies \eqref{no.2}  and \eqref {ms2b}.  The  relation
\eqref {v.6} is then immediate from their definition and Theorem
\ref{re.2}.
 The \eqref{v.7} and \eqref{v.8}  are satisfied by construction of
the $m^{(N)}$.
    Theorem \ref{1.1} is then proved.
\end {proof}

\medskip

\section{Appendix}
\subsection*{A.1: The Dirichlet--Neumann operator}   We  recall  in this section the main
properties of the  Dirichlet--Neumann operator which we have been
using through  the paper. Some of these  results were already
presented  in the  Appendix of \cite {CCO}. Let $ G(\xi, \eta) $
be  the Green's function in $\Omega$  defined in \eqref {2.1000}
and verifying \eqref {NY.4}. To define the  Dirichlet--Neumann
operator   we consider the following single layer potentials.
Given a smooth function $h$ defined on $\G \in \MM$,  consider
the {\it single layer potential}
\[
\phi_h(\xi) = \int_\G G(\xi,\eta)h(\eta)\dse,
\]
where $\dse$ denotes the arc length measure along $\G$.  The
function $\phi_h$ satisfies a Neumann  boundary condition  on
$\partial\Omega$, and also the equation
\[
\Delta \phi_h (\xi) = h(\xi)  -{1\over |\Omega|}\int_\G
h(\eta)\dse\  \quad \xi \in \Gamma.
\]
The curve $\G$ separates $\Omega$ in two subsets. We denote by
$\Om_\G^-$  the interior of $\G$ and  by $\Om_\G^+$ its exterior.
We denote by $n$ the  unit outer  normal to $\Om_\G^-$. There is
a discontinuity in   the normal derivatives of $\phi_h$ across
$\G$, so,
 we have that
\begin{equation}
\label{pota2} h(\xi) =  \frac 12 \left [ \partial_\nu   \phi_h
 \right]_{\G}(\xi), \end{equation}
where the right-hand side denotes the jump of the normal
derivatives at $\xi \in \G$, i.e.
\[
\left [\partial \nu \phi_h
 \right]_{\G}(\xi) =
\left( \partial_\nu \phi_h  \right)_{\Om_\G^+} (\xi) -
\left( \partial_\nu \phi_h \right)_{\Om_\G^-} (\xi).
\]
This is a well known result from potential theory  \cite{GT}.
For $\xi$ away from $\G$,
\[
\Delta \phi_h(\xi) = - {1\over |\Omega|}\int_\G h(\eta)\dse .\]
Thus, the single layer potential is harmonic away from $\G$ if and
only if $\int_\G h(\xi)\dsx = 0$. Otherwise, it is subharmonic or
superharmonic, according to whether $-\int_\G h(\xi)\dsx$ is
positive or negative. Every continuous function $\phi$ harmonic
away from $\G$, satisfying the Neumann boundary condition and the
following relation
\begin{equation}
\label{pota3} \int_\Om \phi(\xi){\rm d}\xi = 0, \end{equation} is
the single layer potential of a uniquely determined function $h$
defined on $\G$  and satisfying
\begin{equation}
\label{pota4} \int_\G h(\xi)\dsx = 0. \end{equation} Indeed, if
$\phi_h$ is such a single layer potential, then from
\eqref{NY.4}, we get that $\int_\Om\phi_h(\xi){\rm d}\xi = 0$.

On the other hand, let $\phi$ be any continuous function that is
harmonic on $\Om_\G^-$ and $\Om_\G^+$, and which satisfies
\eqref{pota3}. Let us define $h$ in $\G$ by
\[
h(\xi) =  \frac 12 \left [ \partial_\nu    \phi
 \right]_{\G}(\xi),
\]
and refer to this as the Neumann data for $\phi$. By the
divergence theorem we obtain
\[
2 \int_\G  h(\xi) \dsx = \int_\G \partial_\nu  \phi^+ \dsx  -
\int_\G \partial_\nu  \phi^- \dsx =  - \int_{\Om \setminus \G
}\Delta \phi {\rm d}\xi   = 0 ,
\]
where $\phi^{\pm}$ denotes the restriction of $\phi$ in
$\Om^{\pm}$. Hence, $h$ satisfies \eqref{pota4}.

Notice that  $\phi - \phi_h$ satisfies the Neumann boundary
conditions and
\[
\left [ \partial_\nu  (\phi- \phi_h)
 \right]_{\G}(\xi) = 0 .
\]
This means that $\phi - \phi_h$ is a constant. Since the integral
is zero then $\phi  = \phi_h$. This proves the one to one
correspondence between single layer potentials of functions $h$
satisfying \eqref{pota4}, and continuous functions $\phi$ that
are harmonic on $\Om_\G^-$ and $\Om_\G^+$, and satisfy
\eqref{pota3}.

Next, given a continuous function $\phi$ that is harmonic on
$\Om_\G^-$ and $\Om_\G^+$, whether or not \eqref{pota3} is
satisfied, we define the function $g$ on $\G$ by $g := \phi|_\G$.
We naturally refer to $g$ as the Dirichlet data for $\phi$. The
Neumann data is $\displaystyle{\left [ \partial_\nu    \phi
 \right]_{\G}}$.
The Dirichlet--Neuman operator ${\mathcal T}_\G$ is defined by
\begin{equation}
\label{roma4} {\mathcal T}_\G g = \frac 12 \left [ \partial_\nu  \phi
 \right]_{\G},  \end{equation}
where $\phi$ is the continuous function that is harmonic in
$\Om_\G^-$ and $\Om_\G^+$, with $\phi|_\G = g$.

A simple argument shows that ${\mathcal T}_\G$ is a positive
Hermitian operator. Indeed, let $\psi$ be continuous on $\Om$,
and  harmonic on $\Om_\G^-$ and $\Om_\G^+$, with $\psi|_\G = h$.
Then it follows that
\[
\begin{split}
2\int_\G h{\mathcal T}_\G g {\rm d}s &=  \int_\G \psi \left [ 
 \partial_\nu   \phi
 \right]_{\G} {\rm d}s\\
&= - \int_{\Om_\G^+} \nabla \cdot(\psi (\nabla \phi)) {\rm d}\xi -
\int_{\Om_\G^-} \nabla \cdot(\psi (\nabla \phi)) {\rm d}\xi\\
&= - \int_\Om \nabla \psi \cdot \nabla \phi {\rm d}\xi
.\end{split}
\]
Taking $h=1$, so that $\psi = 1$, we further see that the range
of ${\mathcal T}_\G$ is orthogonal to the constants. We let
${\mathcal T}_\G $ denote the Friedrichs extension of ${\mathcal
T}_\G $. It is easy to see, and well known, that the  form domain
of ${\mathcal T}_\G $ is the Sobolev space $H^{1/2}(\G)$, and that
the kernel consists exactly of the constants. There is an
explicit formula for the inverse of ${\mathcal T}_\G $ restricted
to the orthogonal complement of the constants; we denote this by
${\mathcal S}_\G $. Indeed, let $v$ be any function on $\G$ with
$\int_\G v(s){\rm d}s = 0$. Since the single layer potential
$\phi_v$ for $v$ has Neumann data $v$, all we need to do is to
subtract a constant to make this function orthogonal to the
constants on $\G$, instead of being orthogonal to the constants on
$\Om$. Therefore, the inverse ${\mathcal S}_\G$ is given by
\begin{equation}
\label{sform} {\mathcal S}_\G v(\xi) = \int_\G
G(\xi,\eta)v(\eta)\dse - {1\over |\G|}\int_\G \int_\G
G(\xi,\eta)v(\eta)\dse\dsx\;\;\;\;\xi \in \G .
\end{equation} It is easily checked that the inverse operator is self adjoint on
the orthogonal complement of the constants. Now let $h$ be an
arbitrary smooth function on $\G$ satisfying $\int_\G h(s){\rm
d}s = 0$, and consider the single layer potential
\[
\phi(\xi) = \int_\G G(\xi,\eta)h(\eta)\dse \;\;\;\;\xi \in \Om .\]
In general, the Dirichlet data for $\phi$ do not integrate to
zero on $\G$ and hence, are not directly related to the Neumann
data through the Dirichlet--Neumann operator. However, we can
correct this by subtracting a constant and defining the function
\[
\tilde \phi(\xi) = \phi(\xi) - {1\over |\G|}\int_\G
\phi(\eta)\dse.
\]
Then obviously we have
\begin{equation}
\label{ma.3}
 \begin{split} & \tilde \phi|_\G  = {\mathcal S}_\G h, \\ &
h = {\mathcal T}_\G \tilde \phi . \end{split} \    \end{equation}
We can now express the vector field $V$ driving the
Mullins--Sekerka flow as
\begin{equation}
\label{ma.2} V = {\mathcal T}_\G \left(K - {1\over |\G|}\int_\G
K(s){\rm d}s\right). \end{equation}

We close by establishing notation for the two harmonic extension
operators that will arise throughout what follows:

The {\it Neumann harmonic extension operator} ${\mathcal
E}_{\G,N}$ is defined by
\begin{equation}
\label{enform} \left ( {\mathcal E}_{\G,N} v\right ) (\xi) =
\int_\G G(\xi,\eta)v(\eta)\dse - {1\over |\G|}\int_\G \int_\G
G(\xi,\eta)v(\eta)\dse\dsx \quad \xi \in \Om ,
\end{equation}
where $v$ is a function on $\G$ satisfying
\[
\int_\G v(\xi)\dsx = 0 .
\]
Notice that $\left ( \NExt v\right ) (\xi)$ is the unique
function that is continuous on $\Om$, harmonic  on $\Om
\backslash \G$ satisfying Neumann boundary conditions on
$\partial \Om$, with Neumann data $v$, and with zero integral
over $\G$.

The {\it Dirichlet harmonic extension operator} ${\mathcal
E}_{\G,D}$ is defined by setting $  {\mathcal E}_{\G,D} g (\xi)$
to be the harmonic function $\phi$ on $\Om\backslash \G$ with
Neumann boundary conditions on $\partial \G$, and $\phi|_\G = g$.
Here, there is no restriction on the integral of $g$ over $\G$.
Naturally, the Dirichlet extension can be expressed in terms of
the Neumann extension and the Dirichlet--Neumann operator.
Relations \eqref{sform} and \eqref{enform} give that
\begin{equation}
\label{extrel}
  {\mathcal E}_{\G,D} g  (\xi) = {\mathcal E}_{\G,N}\left({\mathcal T}_\G
\left(g -  {1\over |\G|}\int_\G g(\eta)\dse\right)\right) (\xi)  +
{1\over |\G|}\int_\G g(\eta)\dse \;\;\;\;\xi \in \Om
.\end{equation} \vskip0.5cm \nada {
\subsection*{A.3: Basic properties of potential theory.}
We recall some basic fact about potential theory that we have
been using trough the paper.

 Under the compatibility condition $\displaystyle{\int_\Omega
f(\xi) d \xi =0}$, the unique solution of the
 equation
\begin{equation}
\label{2.10} \begin {split}
 & \Delta v(\xi)  = f(\xi)\;\;\;\;
\hbox{for}\;\;\;\;\xi \in   \Omega, \cr &  
 \partial _\nu  v = 0 \qquad \hbox {on} \quad \partial \Om
\end  {split}
\end{equation}
satisfying $\displaystyle{\int_\Omega v(\xi) d \xi =0}$, is given
by the volume potential
\begin{equation}
\label{2.114}
 v(\xi) =  \int_\Omega  G(\xi, \eta) f(\eta) d \eta  .       \end{equation}
 In \eqref {2.10}  $\nu$ is the unit outward normal to $\partial \Om$.
All the other solutions with Neumann boundary conditions differ
from this one by a constant. We are mainly concerned on   the
potential which are the sum of single layer potential and volume
potential.

Take
$$ u (\xi) = \int_\G G(\xi,\eta)h(\eta)\dse +   \int_\Omega  G(\xi, \eta) f(\eta) d \eta, \quad  \xi \in \Om  $$
Then $ u$ solves
\begin{equation}
\label{2.10a} \begin {split}
 & \Delta u(\xi) = -\frac 1 {|\Om|}  \int_\G h(\eta)\dse +f(\xi)   -\frac 1 {|\Om|}  \int_\Omega  f(\eta) d \eta
\hbox{for}\;\;\;\;\xi \in   \Omega \setminus \G, \cr & h(\eta) =
 \frac 12 \left [  \partial_\nu \phi_h \right]_{\G}(\xi)
\cr & \partial_\nu v = 0 \qquad \hbox {on}
\quad \partial \Om
\end  {split}
\end{equation}
}

 \vskip0.5cm
\subsection*{A.2: The  expansion in $\e$ of the Laplacian in
local coordinates}  \label {A.2} Let  $ f (z, s)$, with $z=\frac
{d}{\e}$,  be a $C^2 $ function from $ \R \times \G$ to $\R$.
Then, in the two dimensional case, we have that
\begin{equation}
\label{C.11}
 \begin{split}  \e^2 \Delta f  (z, s)&= \frac 1{ 1-K(s)\e z}
 \left \{ \left ( ( 1-K(s)
\e z) f_z \right )_z +\e^2 \left (\frac{ f_s}{ 1-K(s) \e z}
\right)_s \right \}
\\ &=
 f_{zz} - \e K(s)f_z \frac 1{ 1-K(s) \e z} +\e^2\frac{ f_{ss}}{( 1-K(s)
\e z)^2}+ \e^3 f_s \frac{ \frac d {ds } K(s) z }{( 1-K(s) \e
z)^3}.  \end{split}
   \end{equation}
Recalling that for $|x| <1$ it holds that
\[
 \frac 1 {(1-x)} = \sum_{n=0}^\infty x^n,\;\;\;\;\;
  \frac 1 {(1-x)^2} =
\sum_{n=0}^\infty n x^{n-1},\;\;\;\;\;
 \frac 1 {(1-x)^3}  =
\frac 12 \sum_{n=0}^\infty n(n-1) x^{n-2}, \]    we may rewrite
\eqref{C.11}  as follows
\begin{equation}
\label{C.100}  \e^2 \Delta f
 = f_{zz} + \sum_{n=0}^\infty \e^{n+1}  \left \{ a_{n+1}(z,s)     f_z + b_{n+1}(z,s)
f_{ss} + c_{n+1}(z,s) f_s \right \},
\end{equation}
where
\begin{equation}
\label{coef}  \begin{split} & a_{n+1}(z,s)=-K^{n+1}(s) z^n , \\ &
 b_{n+1}(z,s)=n K^{n-1}(s) z^{n-1}   , \\ &
 c_{n+1}(z,s)=\frac 12 n(n-1)  z^{n-1}   K^{n-2}(s)  \frac d {ds } K(s)  .  \end{split}  \end{equation}

 \vskip0.5cm
\subsection*{A.3: proof of Theorem \ref {thFaprox}}  The proof  goes very much as in \cite [Theorem 2.1] {A-B-C}.  Hence we outline only those points where the presence of the $G_1$ and $G_2$  makes a  difference. 
First of all  the constructed  functions $m^{(N)} (t)$, $t \in [0,T]$ satisfy the requirements needed to apply the spectral estimates    proven by
\cite {AF}  and  \cite {Chen2}.    Namely,  see \eqref {e1},  \eqref {e2}  and Lemma \ref {42},
we can always write  for $ \xi \in \Omega$ and  $\in [0,T]$
 $$ m^{(N)}(\xi,t) =  m_0  \Big{(}\frac {d(\xi,\G^{(N)}_t)} \e
\Big{)} +  \e    \tilde h_1 \Big{(}{d(\xi,\G^{(N)}_t)\over \e},
s(\xi,\G^{(N)}_t)\Big{)}  + \e^2  q^\e (\xi,\G^{(N)}_t)
     + \e  \phi^\e (\xi,\G^{(N)}_t), $$
    where  $m_0$ is given in \eqref {e1a},  $\tilde h_1 (\cdot, s)$  is the function determined in  Lemma \ref {42}  which is even  as function of $z \in \R$ for any $s \in \G^{(N)}_t$,
 equal to $ 0 $ in $\Om \setminus \NN
  (\e_0)$ and when $d(\xi,\G^{(N)}_t)=0$.    We denote by $\e^2  q^\e (\xi,\G^{(N)}_t)$ the remaining functions in the expansion of $m^{(N)}$ which are equal to zero in $ \Omega \setminus \NN(\e_0)$
  and by $ \e  \phi^\e (\xi,\G^{(N)}_t)$  the corrections to $\pm 1$ in $ \Omega \setminus \NN(\e_0)$.
 Recall 
$m^{(N)}(\cdot ,t)$  are $C^\infty(\Omega)$ for any $t \in [0,T]$.  We immediate have, since $\bar m (\cdot)$ is odd  while  $h_1 (\cdot,s)$ is  even,  
$$ \int_{\R} \tilde h_1(z,s) (\bar m'(z))^2 f''(\bar m(z)) dz = 6 \int_{\R} \tilde h_1(z,s) (\bar m'(z))^2  \bar m(z) dz =0, \qquad \forall s \in \G^{(N)}_t.  $$
This  is one of the requirement needed to apply the spectral estimates.  The  remaining requirements
are immediately satisfied by the  smoothness of $m^{(N)}(\cdot ,t)$ and by the fact that the $\Phi_j$, $ j=1, \dots, N$   
in the expansion of $ m^{(N)}$ satisfy a  global Lipschitz bound independent on $\e$. 

Then one proceeds as in \cite {A-B-C}. 
Write \eqref{1a}  as the  the following :
\begin{equation}\label{sysch}
\begin{split}
&\partial_tm^\eps=\Delta\mu^\eps+G_1\;\;\;\;\mbox{in}\;\;\Omega_T,\\
&\mu^\eps=-\eps\Delta m^\eps+\frac{1}{\eps}f(m^\eps)-G_2\;\;\;\;\mbox{in}\;\;\Omega_T,\\
&m^{\eps}(\xi,0)=m^{\eps}_0(\xi)\;\;\;\;\xi\in\;\;\Omega,\\
&\partial_\nu m^{\eps}=\partial_\nu\Delta
m^{\eps}=0\;\;\;\;\mbox{on}\;\;\partial\Omega,
\end{split}
\end{equation}
  and \eqref{ch0N} the  the following:
\begin{equation}\label{syscha}
\begin{split}
&\partial_tm^{(N)}=\Delta\mu^{(N)}+\sum_{j=0}^{N-1}\eps^jG_{1,j}\;\;\;\;\mbox{in}\;\;\Omega_T,\\
&\mu^{(N)}=-\eps\Delta m^{(N)}+\frac{1}{\eps}f(m^{(N)})-\sum_{j=0}^{N-1}\eps^jG_{2,j}+R^{(N)}\;\;\;\;\mbox{in}\;\;\Omega_T,\\
&m^{(N)}(\xi,0)=m^{\eps}_0(\xi)\;\;\;\;\xi\in\;\;\Omega,\\
&\partial_\nu m^{(N)}=\partial_\nu\Delta
m^{(N)}=0\;\;\;\;\mbox{on}\;\;\partial\Omega.
\end{split}
\end{equation}
Define
$R:=m^\eps-m^{(N)}$. Then integrating $R$ in space,  by \eqref{sysch} and \eqref{syscha} and the fact that
$R(\xi,0)=0$, we obtain
\begin{equation}\label{thm1}
\begin{split}
\int_\Omega
R(\xi,t)d\xi&=\int_\Omega\Big{(}m^\eps-m^{(N)}\Big{)}d\xi=\int_\Omega  d\xi \int_0^t\partial_s\Big{(}m^\eps-m^{(N)}\Big{)}ds \\
&=\int_\Omega d\xi \int_0^t\Big{(}\Delta\mu^\eps+G_1-\Delta\mu^{(N)}-\sum_{j=0}^{N-1}\eps^jG_{1,j}\Big{)}ds \\
&=\int_0^t ds
\int_\Omega\Delta\Big{[}\mu^\eps-\mu^{(N)}\Big{]}d\xi +
\e^N \int_0^t ds \int_\Omega\Big{[} G_{1,N}\Big{]}d\xi ds\\
&=\int_0^t ds
\int_\Omega\Delta\Big{[}\mu^\eps-\mu^{(N)}\Big{]}d\xi =0,
\end{split}
\end{equation}
since $\partial_\nu[\mu^\eps-\mu^{(N)}]=0$ on $\partial\Omega$,
 and \eqref {D.10} holds.
Note that we need (and used for the above) $\partial_\nu
G_2=\partial_\nu R^{(N)}=0$ on $\partial\Omega$ in order to have
$\partial_\nu[\mu^\eps-\mu^{(N)}]=0$ on $\partial\Omega$.
 In addition, we need the
residual $R^{(N)}$ satisfying a Neumann condition i.e.
$$\partial_\nu R^{(N)}=0\;\;\;\;\mbox{on}\;\;\partial\Omega$$
which is true by construction, cf. system \eqref{v.5} and the b.c.
on system solutions in the statement of Theorem \ref{1}.

Hence, since by \eqref{thm1}, for any $t\in[0,T]$ $\int_\Omega
R(\xi,t)d\xi=0$ then there exists unique $\psi(\xi,t)$ such that
\begin{equation}\label{thm2}
\begin{split}
&-\Delta\psi(\cdot,t)=R(\cdot,t)\;\;\;\;\mbox{in}\;\;\Omega,\\
&\partial_\nu\psi(\cdot,t)=0\;\;\;\;\mbox{on}\;\;\partial\Omega,\\
&\int_\Omega\psi(\cdot,t)=0\;\;\;\;\mbox{for any}\;\;t\in[0,T].
\end{split}
\end{equation}
At this point one can continue the proof as in \cite {A-B-C}. 
\qed

\section*{Acknowledgments}
DA  and GK are supported by `Aristeia' (Excellence) grant, ${\rm
\Sigma\Pi A}$ 00086. GK has been partially supported by the
European Union’s Seventh Framework Programme
(FP7-REGPOT-2009-1) under grant agreement no. 245749
‘Archimedes Center for Modeling, Analysis and Computation’
(University of Crete, Greece).  EO has been supported by  MIUR - PRIN 2011-2013. 


\begin{thebibliography}{10}
\bibitem{A-B-C}
{\sc N.~D. Alikakos, P.~W. Bates and X. Chen}, {\em Convergence of
the Cahn-Hilliard Equation to the Hele-Shaw Model}, Arch. Rat.
Mech. Anal., 128 (1994), pp.~165--205.

\bibitem{AF}
{\sc N.~D. Alikakos and G. Fusco}, {\em The spectrum of the
Cahn-Hilliard operator for generic interfaces in higher space
dimensions},  Indiana University Math. J., 42 (1993),
pp.~637--674.

\bibitem{Antal}
{\sc T. Antal, M. Droz, J. Magnin and Z. R\'acz}, {\em Formation
of Liesengang Patterns: A Spinodal Decomposition Scenarion}, Phys.
Rev. Lett., 83 (1999), pp.~2880--2883.

\bibitem{abak}
{\sc D.~C. Antonopoulou, P.~W. Bates, G.~D. Karali}, {\em Motion
of a droplet for the mass conserving Stochastic Allen-Cahn
equation}, preprint.

\bibitem{ABK}
{\sc D.~C. Antonopoulou, D. Blomker, G.~D. Karali}, {\em Front
Motion in the One-Dimensional Stochastic Cahn-Hilliard Equation},
SIAM J. Math. Anal., 44-5 (2012), pp.~3242--3280.

\bibitem{AK}
{\sc D.~C. Antonopoulou, G.~D. Karali}, {\em Existence of solution
for a generalized Stochastic Cahn-Hilliard Equation on convex
domains}, Discrete Contin. Dyn.
 Syst. B, 16(1) (2011), pp.~31--55.

\bibitem{AKK}
{\sc D.~C. Antonopoulou, G.~D. Karali, G.~T. Kossioris}, {\em
Asymptotics for a generalized Cahn-Hilliard equation with forcing
terms}, Discrete Contin. Dyn. Syst. A, 30(4) (2011),
pp.~1037--1054.


\bibitem{Cahn}
{\sc J.~W. Cahn, J.~E. Hilliard}, {\em Free Energy of a
Nonuniform System, I. Interfacial Free Energy} J. Chem. Phys., 28
(1958), pp.~258--267.

\bibitem{Cahn-alone}
{\sc J.~W. Cahn, J.~E. Hilliard}, {\em Free Energy of a
Nonuniform System II, Thermodynamic basis}, J. Chem. Phys., 30
(1959), pp.~1121--1124.

\bibitem{Weber}
{\sc C. Cardon-Weber}, {\em Cahn-Hilliard stochastic equation:
existence of the solution and of its density}, Bernoulli, 7(5)
(2001), pp.~777--816.

\bibitem{CCO}
{\sc E.~A. Carlen, M.~C. Carvalho and E. Orlandi}, {\em
Approximate Solution of the Cahn-Hilliard Equation via
Corrections to the Mullins-Sekerka Motion}, Arch. Rat. Mech.
Anal., 178 (2005), pp.~1--55.


\bibitem{Chen2} {\sc X. Chen}, {\em  Spectrum of the Allen Cahn,
Cahn-Hilliard and phase field equations for generic interface},
Comm. Partial Diff. Eqns., 19 (1994), pp.~1371--1395.

\bibitem{C2} {\sc X. Chen}, {\em Global asymptotic
limit of solutions of the Cahn-Hilliard Equation}, J. Differential
Geom., 44 (1996), pp.~262--311.

\bibitem{cook}
{\sc H. Cook}, {\em Brownian motion in spinodal decomposition},
Acta Metallurgica, 18 (1970), pp.~297--306.

\bibitem{Debu2}
{\sc G. Da Prato, A. Debussche}, {\em Stochastic Cahn-Hilliard
Equation}, Nonlin. Anal. Th. Meth. Appl., 26(2) (1996),
pp.~241--263.



\bibitem{El-Z}
{\sc C.~M. Elliott, S. Zheng}, {\em On the Cahn-Hilliard
equation}, Arch. Rat. Mech. Anal., 96 (1986), pp.~339--357.

\bibitem{EN}
{\sc J. Escher, Y. Nishiura}, {\em Smooth unique solutions for a
modified Mullins-Sekerka model arising in diblock copolymer},
Hokkaido Mathematical Journal, 31 (2002), pp.~ 137-149

\bibitem{ES}
{\sc J. Escher, G. Simonett}, {\em Classical solutions of
multidimensional Hele-Shaw models}, SIAM J. Math. Anal., 28
(1997), pp.~1028--1047.

\bibitem{ref26}
{\sc L.~C. Evans}, Partial Differential Equations, American
Mathematical Society, 1998.

\bibitem{Fun1}
{\sc T. Funaki}, {\em The scaling limit for a Stochastic PDE and
the separation of phases}, Probab. Theory Relat. Fields., 102
(1995), pp.~221--288.

\bibitem {GT}  {\sc D. Gilbarg and N.~S. Trudinger},
Elliptic partial differential equations of second order,
Springer-Verlag (1977)

\bibitem{Gurtin}
{\sc M.~E. Gurtin}, {\em Generalized Ginzburg-Landau and
Cahn-Hilliard equations based on a microforce balance}, Physica
D, 92 (1996), pp.~178--192.

\bibitem{Hohen}
{\sc P.~C. Hohenberg and B.~I. Halperin}, {\em Theory of dynamic
critical phenomena}, J. Rev. Mod. Phys., 49 (1977), pp.~435--479.

\bibitem{KV00}
{\sc M.~A. Katsoulakis and D.~G. Vlachos}, {\em From microscopic
interactions to macroscopic laws of cluster evolution}, Phys. Rev.
Letters, 84 (2000), pp.~1511--1514.

\bibitem{Kitahara}
{\sc K. Kitahara, Y. Oono and D. Jasnow}, {\em Phase separation
dynamics and external force field}, Mod. Phys. Letters B, 2
(1988), pp.~765--771.


\bibitem{Pego}
{\sc R.~L. Pego}, {\em Front migration in the non-linear
Cahn-Hilliard equation}, Proc. R. Soc. Lond., (A) 422 (1989),
pp.~261--278.

\bibitem{C-H-forcing}
{\sc Quan-Fang Wang, Shin-ichi Nakagiri}, {\em Weak solutions of
Cahn-Hilliard equations having forcing terms and optimal control
problems}, Mathematical models in functional equations (Kyoto,
1999),
 S\={u}rikaisekikenky\={u}usho  K\={o}ky\={u}roku,  1128 (2000), pp.~172--180.

\end{thebibliography}
\end{document}